\documentclass[11pt, a4paper]{amsart}

\usepackage{amsfonts,amsmath,amssymb, amscd,fullpage,enumitem}
\usepackage{stmaryrd}
\usepackage[all]{xy}

\newtheorem{theorem}{Theorem}[section]
\newtheorem{lemma}[theorem]{Lemma}
\newtheorem{proposition}[theorem]{Proposition}
\newtheorem{corollary}[theorem]{Corollary}

\theoremstyle{definition}
\newtheorem{definition}[theorem]{Definition}

\theoremstyle{remark}
\newtheorem{rem}[theorem]{Remark}
\newtheorem{rems}[theorem]{Remarks}
\newtheorem{example}[theorem]{Example}

\newcommand\pf{\begin{proof}}
\newcommand\epf{\end{proof}}

\newcommand\Oo{\mathcal{O}}

\DeclareMathOperator{\Hom}{Hom}
\DeclareMathOperator{\SL}{SL}
\DeclareMathOperator{\Aut}{Aut}

\DeclareMathOperator{\Ker}{Ker}

\DeclareMathOperator{\id}{id}

\numberwithin{equation}{section}

\title{Some isomorphism results for graded twistings of function algebras on finite groups}

\begin{document}

\author{Julien Bichon}
\address{ Universit\'e Clermont Auvergne, CNRS, LMBP, F-63000 CLERMONT-FERRAND, FRANCE}
\email{julien.bichon@uca.fr}

\author{Maeva Paradis}
\email{maeva.paradis@uca.fr}

\begin{abstract}
We provide isomorphism results for Hopf algebras that are obtained as graded twistings  of function algebras on finite groups by cocentral actions of cyclic groups. More generally, we also consider the isomorphism problem for finite-dimensional Hopf algebras fitting into abelian cocentral extensions. We  apply our classification results to a number of concrete examples involving special linear groups over finite fields, alternating groups and dihedral groups.
\end{abstract}

\maketitle

\tableofcontents

\section{Introduction}

Hopf algebras are useful and far-reaching generalizations of groups. In the semisimple (hence finite-dimensional) setting, the framework that is the closest from the one of finite groups, all the known examples arise from groups via a number of sophisticated constructions, and a general fundamental question \cite[Problem 3.9]{am} is whether any semisimple Hopf algebra is ``group-theoretical'' in an appropriate sense. An answer to the above question, positive or not, still would leave open the hard problem of the isomorphic classification of the ``group-theoretical'' Hopf algebras.
This paper proposes contributions to this classification problem, mainly for the class of Hopf algebras that are obtained as graded twisting of function algebras of finite groups.

The graded twisting of Hopf algebras, which differs in general from the familiar Hopf $2$-cocycle twisting construction  \cite{doi},  was introduced in \cite{bny16}, and is the formalization of a construction in \cite{ny} that solved the quantum group realization problem of the Kazhdan-Wenzl categories \cite{kw}. The initial data is that of a graded Hopf algebra $A$, acted on by a group $\Gamma$. The resulting twisted Hopf algebra then has a number of pleasant features related to those of initial one. Among those features, the following one \cite{bny16,bny18} is of particular interest: if $A=\mathcal O(G)$ is the coordinate algebra on a linear algebraic group $G$ and $\Gamma$ has prime order, then all the noncommutative quotients of the graded twisted Hopf algebra again are graded twist of $\Oo(H)$, for a well-chosen ``admissible'' closed subgroup $H\subset G$. This applies in particular to $\Oo_{-1}({\rm SL}_2(\mathbb C))$, whose noncommutative quotients have been discussed and classified in \cite{pod,bina}. 
The results in \cite{bny16,bny18}, however, leave open the question of the isomorphic classification of the Hopf algebras that are obtained by graded twisting, and this is precisely the problem that we discuss in this paper. 

We prove 3  isomorphism results for graded twisting of Hopf algebras of functions on finite groups. 
These results all have in common strong cohomological assumptions on the underlying group, which we believe to be difficult to overcome to obtain general results, but yet are broad enough to cover a number of interesting cases. Namely, we obtain classification results for Hopf algebras that are graded twists of
\begin{enumerate}
 \item $\mathcal O({\rm SL}_n(\mathbb F_q))$ by $\mathbb Z_m$, where $q$ is a power of a prime number, $m={\rm GCD}(n,q-1)$ is prime and $(n,q)\not\in\{(2,9),(3,4)\}$ (see Theorem \ref{thm:slzp});
\item $\Oo(\widetilde{A_n})$ by $\mathbb Z_2$, where $\widetilde{A_n}$ is the unique Schur cover of the alternating group $A_n$, with $n\not=6$ (see Theorem \ref{thm:an}); 
\item $\Oo(\widetilde{S_n})$ by $\mathbb Z_2$, where $\widetilde{S_n}$ is any of the two Schur covers of the symmetric group $S_n$, with $n\not=6$ (see Theorem \ref{thm:sn}).
\end{enumerate}

While the first two isomorphism theorems (Theorem \ref{thm:firstmain} and Theorem \ref{thm:firstmainbis}) are obtained rather directly and early in the paper (in Section 3), the third one (Theorem \ref{thm:secondmain}) is obtained by considering the more general problem of the classification of the Hopf algebras fitting into an abelian cocentral extension. This is a classical topic in the field, which has been quite studied and very successful to obtain several classification results \cite{mas95,nat,kas00}. Most of our analysis in Section 4 is thus  well-know to specialists, but we feel that certain formulations and our focus on extensions that are universal bring some novelty, and we get as applications some results in this framework that seem to be new. Indeed we obtain classification results (i.e. parameterizations by concrete and explicitly known group-theoretical data) for noncommutative Hopf algebras $A$ fitting into an abelian cocentral extension $k \to \Oo(H)\to A \to k\mathbb Z_m$ in the following cases:
\begin{enumerate}
 \item $H={\rm PSL}_2(\mathbb F_p)$, with $p$ odd prime and $m=2$;
\item $H=A_n$, with $n=5$ or $n\geq 8$ and $m=2$;
\item $H=A_5$, for any $m\geq 1$;
\item $H=S_n$, with $n\not=6$ and $m=2$;
\item $H=D_n$, the dihedral group of order $2n$ with $n$ odd and $m\geq 1$;
\item $H=D_n$ with $n$ even, with the above extension universal and $m=2$;
\item $H=\mathbb Z_p\times \mathbb Z_p$ with $p$ an odd prime and $m$ a power of a prime such that $m|(p-1)$.
\end{enumerate}
Among those examples, it is interesting to note that the one with $D_n$ and $n$ even is certainly the most intricate one, and does not follow from a general result, although the structure of this group is certainly not the richest one.

The paper is organized as follows. Section 2 consists of reminders and preliminaries. In Section 3 we provide our first two isomorphism results for graded twistings of function algebras on finite groups. Section 4 deals with general abelian cocentral extensions and provides our third isomorphism result for graded twistings. The final Section 5 discusses  applications of the previous results to the concrete examples mentioned above. 

\medskip

\noindent
\textbf{Notation and conventions.} We work over a fixed base field $k$, that we assume to be algebraically closed and of characteristic zero. We assume familiarity with the theory of Hopf algebras, for which \cite{mon} is a convenient reference, and we adopt the usual conventions: for example $\Delta$, $\varepsilon$ and $S$ always respectively stand for the comultiplication, counit and antipode of a Hopf algebra, and we will use Sweedler's notation in the standard manner. A slightly less usual convention is that we will assume that Hopf algebras have bijective antipode.  We also assume some familiarity with basic homological algebra, for which \cite{hs,kar} are convenient references, and in particular we will use \cite{kar} as a reference for Schur multiplier computations.  Other specific notations will be introduced throughout the text.

\section{Preliminaries}
This section consists of reminders about cocentral Hopf algebra maps, cocentral gradings, and the graded twisting construction. It also provides a number of simple but useful preliminary results.

\subsection{Cocentral Hopf algebra maps, cocentral gradings}
The concept of cocentral Hopf algebra map is dual to the familiar one of central algebra map. The precise definition is as follows \cite{ad}, see \cite{chi,chka} for extensive discussions on these notions.
\begin{definition}
	\begin{enumerate}
		\item A Hopf algebra map $p: A \to B$ is said to be \textsl{cocentral} if for any $a \in A$, we have $p(a_{(1)})\otimes a_{(2)}= p(a_{(2)}) \otimes a_{(1)}$.
		\item A cocentral Hopf algebra map $p: A \to B$ is said to be \textsl{universal} if for any cocentral Hopf algebra map $q : A \to C$, there exists a unique Hopf algebra map $f : B \to C$ such that $f\circ p=q$.
		\item A Hopf algebra is said to have a \textsl{universal grading group} if there exists a universal cocentral Hopf algebra map $p: A \to k\Gamma$ for some group $\Gamma$, the unique such group $\Gamma$ being called the universal grading group of $A$.
	\end{enumerate}
\end{definition}

\begin{rems}\label{rems:cocentralbasic}
	\begin{enumerate}
		\item If $p: A \to B$ is a cocentral surjective Hopf algebra map, then $B$ is necessarily cocommutative.
		\item Given a Hopf algebra $A$, the existence of a universal cocentral Hopf algebra map $A \to B$ is easily shown as follows. Consider $X$, the linear subspace of $A$ spanned by the elements $$\varphi(a_{(1)})a_{(2)}-\varphi(a_{(2)})a_{(1)}, \ \varphi \in A^*, \ a \in A$$ 
		It is easy to see that $X$ is a co-ideal in $A$, and then the ideal $I$ generated by $X$ is a Hopf ideal in $A$. The quotient Hopf algebra map $p : A \to A/I$ is then universal cocentral. Uniqueness of the universal cocentral Hopf algebra map is obvious from the definition.
		\item If $G$ is a linear algebraic subgroup, denote by $\Oo(G)$ the algebra of coordinate functions on $G$. If $H \subset G$ is a closed subgroup, the restriction map $\Oo(G)\to \Oo(H)$ is cocentral if and only if $H$ is central in $G$: $H \subset Z(G)$, and the restriction map $\Oo(G) \to \Oo(Z(G))$ is universal cocentral.
		\item If a Hopf algebra $A$ is cosemisimple, it is not difficult to see, using the Peter-Weyl decomposition of $A$ (decomposition of $A$ into direct sum  of matrix subcoalgebras), that $A$ has a universal grading group. 
	\end{enumerate}
\end{rems}

The following lemma will be used several times in the text.

\begin{lemma}\label{lemm:semi-univ}
	Let $A$, $B$ be Hopf algebras having the same universal finite cyclic grading group $\Gamma_0$ and suppose given two surjective cocentral Hopf algebra maps $p: A \to k\Gamma$ and $q : B \to k\Gamma$ for some finite cyclic group $\Gamma$, and a Hopf algebra isomorphism $f : A \to B$. Then there exists $u \in \Aut(\Gamma)$ such that $u\circ p = q\circ f$.
\end{lemma}

\begin{proof}
	Let  $p_0: A \to k\Gamma_0$ and $q_0 : B \to k\Gamma_0$ be the universal cocentral Hopf algebra maps. The Hopf algebra map $q_0\circ f : A \to k\Gamma_0$ being cocentral, there exists a unique group morphism $v :  \Gamma_0 \to \Gamma_0$ such that $v\circ p_0=q_0\circ f$. Since $q_0\circ f$ is surjective, so is $v$ and hence $v$ is an automorphism since $\Gamma_0$ is finite. The Hopf algebra maps $p: A \to k\Gamma$ and $q : B \to k\Gamma$ being cocentral and surjective, the universality of $p_0$ and $q_0$ yields surjective group morphisms $w,w' : \Gamma_0 \to \Gamma$ such that $w\circ p_0=p$ and $w'\circ q_0=q$. Let 
	$N=\Ker(w)$ and $N'=\Ker(w')$. We have $|N|= \frac{|\Gamma_0|}{|\Gamma|}=|N'|$, hence the uniqueness of a subgroup of given order in a finite cyclic group yields $N=N'=v(N)$, and  there exists a unique group morphism $u : \Gamma \to \Gamma$ such that $u\circ w=w'\circ v$:
	$$\xymatrix{
A \ar[r]_{p_0} \ar[d]^f  \ar@/^1pc/[rr]^p& k \Gamma_0  \ar@{-->}[d]^v \ar@{-->}[r]_w & k\Gamma	\ar@{-->}[d]^u\\
B \ar[r]^{q_0} \ar@/_1pc/[rr]_q& k \Gamma_0   \ar@{-->}[r]^{w'} & k\Gamma
}$$
We get	$u\circ p= u\circ w\circ p_0=w'\circ v\circ p_0=w'\circ q_0\circ f=q\circ f$, as required.
	\end{proof}

\begin{definition}
	Let $A$ be a Hopf algebra and let $\Gamma$ be a group. A \textsl{cocentral grading} of $A$ by $\Gamma$ consists of a direct sum decomposition $A=\oplus_{g \in \Gamma}A_g$ such that for any $g,h \in \Gamma$ we have \begin{enumerate}
		\item $A_gA_h\subset A_{gh}$ and $1 \in A_e$, 
		\item $\Delta(A_g)\subset A_g \otimes A_g$ and $S(A_g)\subset A_{g^{-1}}$.
	\end{enumerate}
\end{definition}	

Cocentral gradings by $\Gamma$ correspond to cocentral Hopf algebra maps $p : A \to k\Gamma$. Indeed, given  a cocentral Hopf algebra map $p : A\to k\Gamma$, the corresponding grading is defined by 
$$A_{g}=\{a \in A \ | \ a_{(1)}\otimes p(a_{(2)})=a \otimes g=  a_{(2)} \otimes p(a_{(1)}) \}$$
We occasionally denote the set $A_g$ by $A_{g,p}$ to specify the dependence on $p$, in case there is a risk of confusion.
 Conversely, given a cocentral grading by $\Gamma$, the cocentral Hopf algebra map $p : A\to k\Gamma$ is defined by $p_{|A_g}=\varepsilon(-)g$, and is surjective if and only if $A_g\not=\{0\}$ for any $g \in \Gamma$. We will freely circulate from cocentral Hopf algebra maps to cocentral gradings. 

An important property of the cocentral gradings, provided that the corresponding cocentral Hopf algebra map is surjective, is that they are strong: for any $g,h \in \Gamma$, we have $A_gA_h=A_{gh}$ (see e.g. \cite[Proposition 2.2]{bny18} for this well-known fact). Here is a useful application, used in the proof of the forthcoming Lemma \ref{lemm:iso-weak}.



\begin{lemma}\label{lemm:xy=yz}
 Let $p : A \to k\Gamma$ be a cocentral surjective Hopf algebra map. Let $g\in \Gamma$ and let $y,z \in A$ be
 such that $xy=xz$ for any $x \in A_g$. Then $y=z$. 
\end{lemma}

\begin{proof}
Since $A_e=A_{g^{-1}}A_g$, there exist  $x_1,\ldots ,x_m\in A_{g^{-1}}$ and $y_1, \ldots, y_m\in A_g$  such that $1= \sum_{i=1}^mx_iy_i$. Then, using our assumption, we have
$y= \sum_{i=1}^mx_iy_iy=\sum_{i=1}^mx_iy_iz=z$.
\end{proof}

\subsection{Cocentral actions and graded twisting}
The following notion is introduced in \cite{bny16} under the name ``invariant cocentral action". In the present paper, to simplify terminology, we will simply say ``cocentral action".

\begin{definition}
	A \textsl{cocentral action} of a group $\Gamma$ on a Hopf algebra $A$ consists of a pair $(p,\alpha)$ where $p : A \to k\Gamma$ is a surjective cocentral Hopf algebra map and $\alpha : \Gamma \to \Aut_{\rm Hopf}(A)$ is a group morphism, together with the compatibility condition $p\circ \alpha_g=p$ for any $g \in \Gamma$.  
\end{definition}

In the graded picture, the compatibility condition is $\alpha_g(A_h)=A_h$ for any $g,h\in \Gamma$.

\begin{definition}
Given a cocentral action  $(p,\alpha)$ of a group $\Gamma$ on a Hopf algebra $A$, the \textsl{graded twisting} $A^{p,\alpha}$ is the Hopf algebra having $A$ as underlying coalgebra, and whose product and antipode  are defined by
$$\forall a \in A_g, b\in A_h, \ a\cdot b = a\alpha_g(b), \ S(a) = \alpha_{g^{-1}}(S(a))$$ 
\end{definition}

The present definition of a graded twisting differs from the original one in \cite{bny16}, but is equivalent to it: see \cite[Remark 2.4]{bny18}, the underlying algebra structure is that of a twist in the sense of \cite{zha}.

\begin{lemma}\label{lemm:invariantuniv}
	Let $q : A \to B$ be a universal cocentral Hopf algebra map and let $(p,\alpha)$ be a cocentral action of a group $\Gamma$ on  $A$. Then  $q : A^{p,\alpha} \to B$ still is a universal cocentral Hopf algebra map.
\end{lemma}

\begin{proof}
Recall from Remark \ref{rems:cocentralbasic} that we can assume that $q$ is the quotient map $A \to A/I$ where $I$ is the ideal of $A$ generated by $X$, the linear subspace of $A$ spanned by the elements $\varphi(a_{(1)})a_{(2)}-\varphi(a_{(2)})a_{(1)}, \ \varphi \in A^*, \ a \in A$. The space $X$ is as well  
 the linear subspace of $A$ spanned by the elements 
$$\varphi(a_{(1)})a_{(2)}-\varphi(a_{(2)})a_{(1)}, \ \varphi \in A^*, \ a \in A_g, \ g \in \Gamma$$
Let $I'$ be the ideal of $A^{p,\alpha}$ generated by $X$. The computation, for $a \in A_g, b \in A_h, c \in A_r$,
\begin{align*}
a\cdot \left(\varphi(b_{(1)})b_{(2)}-\varphi(b_{(2)})b_{(1)}\right)\cdot c &=  \varphi(b_{(1)})a\alpha_g(b_{(2)})\alpha_{gh}(c)-\varphi(b_{(2)})a\alpha_g(b_{(1)})\alpha_{gh}(c) \\
		& =  a\left(\varphi\alpha_{g^{-1}}(\alpha_g(b_{(1)}))\alpha_g(b_{(2)})-\varphi\alpha_{g^{-1}}(\alpha_g(b_{(2)}))\alpha_g(b_{(1)})\right)\alpha_{gh}(c)
	\end{align*}
 shows that $I'\subset I$. In a symmetric manner, since $ab=a\cdot\alpha_{g^{-1}}(b)$ for $a \in A_g$ and $b \in A$, we have $I \subset I'$ and hence $I=I'$. Therefore the quotient map $q' : A^{p,\alpha} \to A^{p,\alpha}/I'$, which is universal cocentral, equals $q$, and we have our result.
	\end{proof}

Since our main goal is to compare the different Hopf algebras obtained via graded twisting, an obvious thing to do first is to compare the various cocentral actions, and for this the following notion is quite natural.
	
\begin{definition}
	Two cocentral actions $(p,\alpha)$ and $(q,\beta)$ of a group $\Gamma$ on a Hopf algebra $A$ are said to be \textsl{equivalent} if there exist $u\in \Aut(\Gamma)$ and $f\in \Aut_{\rm Hopf}(A)$ such that 
	$$u\circ p=q\circ f \ \textrm{and} \ \forall g\in \Gamma, f \circ \alpha_g  \circ f^{-1} = \beta_{u(g)}$$
\end{definition}	

\begin{lemma}\label{lemm:equi-iso}
 Let $(p,\alpha)$ and $(q,\beta)$ be cocentral actions of a group $\Gamma$ on a Hopf algebra $A$. If  $(p,\alpha)$ and $(q,\beta)$ are equivalent, then the Hopf algebras $A^{p,\alpha}$ and $A^{q,\beta}$ are isomorphic.
\end{lemma}

\begin{proof}
Let $u\in \Aut(\Gamma)$ and $f\in \Aut_{\rm Hopf}(A)$ as in the above definition. The condition $u\circ p=q\circ f$ ensures that $f(A_g)=A_{u(g)}$ for any $g \in \Gamma$. Hence for $a \in A_g$ and $b$ in $A$, we have
$$f(a\cdot b) = f(a\alpha_g(b))= f(a)f(\alpha_g(b))= f(a)\beta_{u(g)}(f(b))= f(a)\cdot f(b)$$
which shows that $f$ is as well a Hopf algebra isomorphism from  $A^{p,\alpha}$ to $A^{q,\beta}$.
	\end{proof}

There is also another weaker notion of equivalence for cocentral actions, as follows.

\begin{definition}
	Two cocentral actions $(p,\alpha)$ and $(q,\beta)$ of a group $\Gamma$ on a Hopf algebra $A$ are said to be \textsl{weakly equivalent} if  there exist $u\in \Aut(\Gamma)$ and a Hopf algebra isomorphism $f : A_{e,p} \to A_{e,q}$ such that 
$$ \forall g\in \Gamma, f \circ ({\alpha_g})_{|A_{e,p}}  \circ f^{-1} = (\beta_{u(g)})_{|A_{e,q}}$$	
\end{definition}

Not surprisingly, equivalent cocentral actions are weakly equivalent.

\begin{lemma}
	Two equivalent cocentral actions $(p,\alpha)$ and $(q,\beta)$ of a group $\Gamma$ on a Hopf algebra $A$ are weakly equivalent.
\end{lemma}

\begin{proof}
Let	$u\in \Aut(\Gamma)$ and $f\in \Aut_{\rm Hopf}(A)$ be such that 
	$u\circ p=q\circ f$ and $\forall g\in \Gamma, f \circ \alpha_g \circ f^{-1} = \beta_{u(g)}$. Then $f(A_{g,p})=A_{u(g),q}$, $\forall g\in \Gamma$, hence $f(A_{e,p})=A_{e,q}$, and the conclusion follows.
	\end{proof}

It is unclear to us whether the existence of a Hopf algebra isomorphism between $A^{p,\alpha}$ and $A^{q,\beta}$ forces the cocentral actions $(p,\alpha)$ and $(q,\beta)$ to be weakly equivalent. However this holds true in the following special situation.

\begin{lemma}\label{lemm:iso-weak}
 Let $A$ be a commutative Hopf algebra having a finite cyclic universal grading group, and let $(p,\alpha)$, $(q,\beta)$ be cocentral actions of a cyclic group $\Gamma$ on $A$. If the Hopf algebras $A^{p,\alpha}$ and $A^{q,\beta}$ are isomorphic, then the cocentral actions $(p,\alpha)$ and $(q,\beta)$ are weakly equivalent. 
\end{lemma}

\begin{proof}
Let $f :  A^{p,\alpha} \to A^{q,\beta}$ be a Hopf algebra isomorphism. By Lemma \ref{lemm:invariantuniv}, we can apply Lemma \ref{lemm:semi-univ} to get  $u\in \Aut(\Gamma)$ such that $u\circ p=q\circ f$, so that $f(A_{g,p})=A_{u(g),q}$ for any $g\in \Gamma$, and in particular $f(A_{e,p})=A_{e,q}$. For $a \in A_g$ and $b \in A_e$, we have
$$f(a\alpha_g(b))=f(a\cdot b)= f(a)\cdot f(b) = f(a)\beta_{u(g)}(f(b))$$
By the commutativity of $A$, we have as well
$$f(a\alpha_g(b))= f(\alpha_g(b)a)=f(\alpha_g(b)\cdot a) = f(\alpha_g(b))\cdot f(a) = f(\alpha_g(b))f(a)=f(a)f(\alpha_g(b))$$
We conclude from Lemma \ref{lemm:xy=yz} that $\beta_{u(g)}(f(b))= f(\alpha_g(b))$, so our cocentral actions are indeed weakly equivalent.
\end{proof}

\subsection{Graded twisting of function algebras} In this subsection we translate in group theoretical terms the notions discussed in the previous subsections when $A=\Oo(G)$, the function algebra on a finite group $G$ (this of course runs as well if we assume that $G$ is a linear algebraic group, but for simplicity we restrict to the finite case). The translations are rather obvious, convenient, and induce a few new notations. As usual, if
$\Gamma$ is group, the dual group ${\rm Hom}(\Gamma,k^\times)$ is denoted $\widehat{\Gamma}$.  If $G$ is a group and $T\subset G$ is a subgroup, we denote by $\Aut_T(G)$ the group of automorphisms of $G$ that preserve $T$, and by $\Aut_T^\circ(G)$ the subgroup of automorphisms that fix each element of $T$.  

\begin{enumerate}
	\item A cocentral action $(p,\alpha)$ of the finite group $\Gamma$ on $\Oo(G)$ corresponds to a pair $(i,\alpha)$ where $i : \widehat{\Gamma} \to Z(G)$ is an injective group morphism and  $\alpha :  \Gamma \to \Aut_{i(\widehat{\Gamma})}^\circ(G)$  a group morphism. We then consider cocentral actions of $\Gamma$ on $\Oo(G)$ as such pairs $(i,\alpha)$, call them cocentral actions on $G$, and denote the corresponding graded twisting $\Oo(G)^{p,\alpha}$ by $\Oo(G)^{i,\alpha}$.
	\item Two cocentral actions $(i,\alpha)$ and $(j,\beta)$ are equivalent if there exist $u\in \Aut(\Gamma)$ and $f\in \Aut(G)$ such that 
	$$i\circ \widehat{u}=f\circ j \ \textrm{and} \ \forall g\in \Gamma, \ f^{-1} \circ \alpha_g  \circ f = \beta_{u(g)}$$
where $\widehat{u} = -\circ u$.
		\item Two cocentral actions $(i,\alpha)$ and $(j,\beta)$ are weakly equivalent if there exist $u\in \Aut(\Gamma)$ and an isomorphism $f : G/j(\widehat{\Gamma}) \to G/i(\widehat{\Gamma})$ such that 
	$\forall g\in \Gamma, f^{-1}\circ \overline{\alpha}_g \circ f = \overline{\beta_{u(g)}}$, where $\overline{\alpha_g}$ and $\overline{\beta_{u(g)}}$  denote the automorphisms of $G/i(\widehat{\Gamma})$ and $G/j(\widehat{\Gamma})$ induced by  $\alpha_g$  and $\beta_{u(g)}$ respectively.
\end{enumerate}

Assuming that the finite group $G$ has a cyclic center, there is a convenient way to describe the equivalence classes of cocentral actions of $\mathbb Z_m$ on $G$, as follows.

For $m$ a divisor of $|Z(G)|$, let $T_m$ be the unique subgroup of order $m$ of $Z(G)$, and let $\mathbb X_m(G)$ be the set of  elements $\alpha_0\in {\rm Aut}_{T_m}^\circ(G)$ such that $\alpha_0^m={\rm id}_G$ modulo the equivalence relation
$$\alpha_0\sim \beta_0 \iff \text{$\exists f \in \Aut_{T_m}(G)$ and $l$ prime to $m$ such that $f^{-1}\circ \alpha_0\circ f=\beta_0^l$ and $f_{|T_m}=(-)^l$}$$ 
For $\alpha_0\in \Aut_{T_m}^\circ(G)$, we denote by $\ddot{\alpha_0}$ its equivalence class in $\mathbb X_m(G)$. We will also denote by    $\mathbb X_m^\bullet(G)$ the set of equivalence classes $\ddot{\alpha_0}$ such that $\alpha_0$ does not induce the identity on $G/T_m$.

\begin{lemma}\label{lem:XmG}
 If $G$ is a finite group with cyclic center and $m$ is a divisor of $|Z(G)|$, we have a bijection
$\mathbb X_m(G) \simeq \{\text{equivalence classes of cocentral actions of $\mathbb Z_m$ on $G$}\}$.
\end{lemma}

\begin{proof}
 Fix a generator $g$ of $\mathbb Z_m$ and an injective group morphism $i : \widehat{\mathbb Z_m}\to Z(G)$ with $T_m=i(\widehat{\mathbb Z_m})$, and associate to 
$\alpha_0\in \Aut_{T_m}^\circ(G)$ the cocentral action $(i,\alpha)$ of $\mathbb Z_m$ on $G$ with $\alpha_g=\alpha_0$. It is clear that for $\alpha_0,\beta_0 \in \Aut_{T_m}^\circ(G)$, we have $\ddot{\alpha_0}=\ddot{\beta_0}$ if and only if the cocentral actions $(i,\alpha)$ and $(j,\beta)$ are equivalent, so we get the announced injective map. 

Start now with a cocentral action $(j,\beta)$ of $\mathbb Z_m$ on $G$. Let $u$ be the automorphism of $\mathbb Z_m$ defined  by $\widehat{u}=i^{-1}\circ j$: $u =(-)^l$ for $l$ prime to $m$. For $l'$ such that $ll'\equiv 1 [n]$, we then see that the cocentral actions 
$(j,\beta)$ and $(i,\beta^{l'})$ are equivalent, and this proves that our map is surjective.
\end{proof}

\subsection{Group-theoretical preliminaries}\label{subsec:groupprelim} This last subsection consists of group theoretical preliminaries. As usual, if $G$ is a group and $M$ is a $G$-module, the second cohomology group of $G$ with coefficients in $M$ is denoted $H^2(G,M)$. We mainly consider trivial $G$-modules (the only exception is in the proof of Lemma \ref{lemm:alga(h)}). If $\tau \in Z^2(G,M)$ is a (normalized) $2$-cocycle, its cohomology class in $H^2(G,M)$ is denoted $[\tau]$, while if $\mu : G \to M$ is a map with $\mu(1)=1$, the associated $2$-coboundary is denoted $\partial(\mu)$.

Our first lemma is certainly well-known. We provide the details of the proof for future use.

\begin{lemma}\label{lemm:autocentral}
 Let $T$ be a central subgroup of a group $G$. There is an exact sequence of groups
$$1 \to \Hom(G/T,T) \to \Aut_{T}(G) \to \Aut(G/T)\times \Aut(T)$$
where the map on the right is surjective when $|H^2(G/T,T)|\leq 2$ (or more generally when the natural actions of $\Aut(G/T)$ and $\Aut(T)$ on $H^2(G/T,T)$ are trivial).
\end{lemma}

\begin{proof}
Since any element in  $\Aut_{T}(G)$ simultaneously restricts to an automorphism of $T$ and induces an automorphism of $G/T$, we get the group morphism on the right. 
Given $\chi \in \Hom(G/T,T)$, define an automorphism $\tilde{\chi}$ of $G$ by $\tilde{\chi}(x) =x\chi(\pi(x))$, where $\pi : G\to G/T$ is the canonical surjection. This defines the group morphism $\Hom(G/T,T) \to \Aut_{T}(G)$ on the left, which is clearly injective and whose image is easily seen to be the kernel of the map on the right.

Put $H=G/T$. By the standard description of central extensions of groups, we can freely assume that $G=H\times_\tau T$ where $\tau \in Z^2(H,T)$ and the product of $G$ is defined	by
$$\forall x,y\in H, \ \forall r,s \in T, \ (x,r)\cdot(y,s) = (xy,\tau(x,y)rs).$$
It is straightforward to check that an element $\alpha \in  \Aut_{T}(G)$ is defined by $\alpha(x,t)=(\theta(x),\mu(x)u(t))$, where  $(\theta, \mu, u)$ is a triple with $\theta \in \Aut(H)$, $u \in \Aut(T)$ and $\mu : H \to T$ satisfying
$$\forall x,y \in H, \ u(\tau(x,y)) \mu(xy)=\mu(x)\mu(y) \tau(\theta(x),\theta(y)). \ \ (\star)$$ 
Under this identification, the composition law in $ \Aut_{T}(G)$ is given by 
$$(\theta,\mu,u)(\theta',\mu',u')=(\theta \circ \theta',\mu\circ\theta'\cdot u\circ\mu',u\circ u').$$
 The map $\Aut_{T}(G) \to \Aut(H) \times \Aut(T)$ in the statement of the lemma is then the projection on the first and third factor, and elements of the kernel are exactly those of the form $({\rm id}_H,\mu,{\rm id}_T)$, where $\mu : H \to T$ is a group morphism. 
 
Assume now the natural actions of $\Aut(H)$ and $\Aut(T)$ on $H^2(H,T)$  by group automorphisms are trivial (which obviously holds when $|H^2(H,T)|\leq 2$).
Let $(\theta,u) \in \Aut(H)\times \Aut(T)$. The cocycles $u\circ \tau$ and $\tau\circ (\theta \times \theta)$ are then cohomologous to $\tau$, and hence there exists $\mu : H \to T$ such that
 $u\circ \tau = \partial(\mu) \tau \circ (\theta \times \theta)$, which is exactly the condition $(\star)$ that allows $(\theta,\mu,u)$ to define an element of $\Aut_T(G)$, and thus the map on the right in our exact sequence is surjective.
\end{proof}

Our second lemma will be used at the end of Section 4.

\begin{lemma}\label{lemm:autocentral2}
	Let $H$ be a finite group, let $T$ be a cyclic group of order $m$, let $\tau\in Z^2(H,T)$ and consider the group $G=H\times_\tau T$. Let $\alpha,\beta \in \Aut_T^\circ(G)$ (i.e. $\alpha_{|T}={\rm id}_T=\beta_{|T}$), and let $\overline{\alpha}$, $\overline{\beta}$ be the induced automorphisms of $H$. Assume that ${\rm Hom}(H,T)=\{1\}$ and that there exists $\theta\in \Aut(H)$ and $l$ prime to $m$ such that 
	\[ \theta \circ \overline{\alpha}\circ \theta^{-1}= \overline{\beta}^l \ {\rm and } \ [\tau]^l=[\tau \circ \theta \times \theta] \ {\in} \ H^2(H,T)\]
	Then there exists $f\in \Aut_T(G)$ such that
	\[ f\circ \alpha \circ f^{-1} =\beta^l \ {\rm and} \ f_{|T}=(-)^l\]  
\end{lemma}

\begin{proof}
Recall from the proof of the previous lemma that the elements of $\Aut_T(G)$ are represented by triples 
$(\theta, \mu, u)$ 	with $\theta \in \Aut(H)$, $u \in \Aut(T)$ and $\mu : H \to T$ satisfying
$ u\circ \tau =\partial(\mu) \tau\circ \theta\times \theta$, with $(\theta, \mu, u)(x,t)=(\theta(x),\mu(x)u(t))$, for $(x,t)\in H\times T$. By assumption, with this convention, we have $\alpha=(\overline{\alpha}, \phi, {\rm id}_T)$ and $\beta=(\overline{\beta},\gamma,{\rm id}_T)$. Let $u$ be the automorphism of $T$ defined by $u=(-)^l$. The assumption $[\tau]^l=[\tau \circ \theta \times \theta] $ thus amounts to $[u\circ \tau]=[\tau \circ \theta \times \theta]$, hence there exists $\mu : H \to T$ such that $\theta$ extends to an automorphism $f=(\theta,\mu,u)$ of $G$.
We have 
\begin{align*} 
f\circ \alpha \circ f^{-1} & = (\theta,\mu,u) (\overline{\alpha},\phi,{\rm id}_T)(\theta,\mu,u)^{-1} \\
&=(\theta\circ \overline{\alpha},\mu\circ \overline{\alpha}\cdot u\circ\phi ,u)(\theta^{-1},u^{-1}\circ((\mu\circ\theta^{-1})^{-1}),u^{-1})\\
& = (\theta\circ \overline{\alpha}\circ \theta^{-1}, \chi,{\rm id}_T)\\
& = (\overline{\beta}^l, \chi,{\rm id}_T)
\end{align*} 
for some $\chi : H\to T$.
Hence $f\circ \alpha \circ f^{-1}$ and $\beta^l$ have the same image under the group morphism on the right in the previous lemma, and the assumption ${\rm Hom}(H,T)=\{1\}$ thus implies that $f\circ \alpha \circ f^{-1} =\beta^l$. We have moreover $f_{|T}=u=(-)^l$, and this finishes the proof.
	\end{proof}
	
To finish this section, we record a last lemma, again to be used in Section 4. It is well known that inner automorphisms act trivially on the second cohomology of a group. Our next lemma is an explicit writing of this fact. The proof is a straightforward verification, but can also be obtained easily from the considerations in the proof of Lemma \ref{lemm:autocentral}.

\begin{lemma}\label{lem:tauinner}
 Let $H$ be a group, let $x\in H$ and let $\tau\in Z^2(H,k^\times)$. Then we have $\tau=\partial(\mu_x)\cdot \tau\circ {\rm ad}(x)\times {\rm ad}(x)$, where $\mu_x(y)=\tau(xy,x^{-1})\tau(x,y)\tau(x,x^{-1})^{-1}$.
\end{lemma}

\section{First results}

We are now ready to state and prove our first two isomorphism results for graded twistings of function algebras on finite groups.

\begin{theorem}\label{thm:firstmain}
	Let $G$ be a finite group with cyclic center, let $(i,\alpha)$ and $(j,\beta)$ be cocentral actions of a cyclic group $\Gamma$ on $G$, and put $H=G/i(\widehat{\Gamma})=G/j(\widehat{\Gamma})$. Assume that $|H^2(H,\widehat{\Gamma})|\leq 2$ (or more generally that the natural actions of $\Aut(H)$ and $\Aut(\widehat{\Gamma})$ on $H^2(H,\widehat{\Gamma})$ are trivial) and that $\Hom(H,\widehat{\Gamma})=\{1\}$. Then the following assertions are equivalent.
	\begin{enumerate}
		\item The Hopf algebras $\Oo(G)^{i,\alpha}$ and $\Oo(G)^{j,\beta}$ are isomorphic.
		\item The cocentral actions $(i,\alpha)$ and $(j,\beta)$ are equivalent.
		\item The cocentral actions $(i,\alpha)$ and $(j,\beta)$ are weakly equivalent.
	\end{enumerate}
\end{theorem}

\begin{proof} First notice that, since $Z(G)$ is cyclic, it has a unique subgroup of a given order, and we have indeed $i(\widehat{\Gamma})=j(\widehat{\Gamma})$. 
	Since $(1)\Rightarrow (3)$ follows from Lemma \ref{lemm:iso-weak} and $(2)\Rightarrow (1)$ follows from Lemma \ref{lemm:equi-iso}, it remains to show that $(3)\Rightarrow (2)$.  

Denote by $f \mapsto \overline{f}$ the group morphism $\Aut_{i(\widehat{\Gamma})}(G) \to \Aut(H)$ of Lemma \ref{lemm:autocentral}.  Fix a generator $g \in \Gamma$ and assume the existence of  $f\in \Aut(H)$ and $u \in \Aut(\Gamma)$  such that $f^{-1}\circ \overline{\alpha_g} \circ f = \overline{\beta_{u(g)}}$. Our assumption on $H^2(H,\widehat{\Gamma})$ ensures, by Lemma \ref{lemm:autocentral}, the existence of $f_0 \in \Aut_{i(\widehat{\Gamma})}(G)$ such that 
$$\overline{f_0}= f \  {\rm and} \ f_{0|i(\widehat{\Gamma})}=i\circ \widehat{u}\circ j^{-1}, \ i.e. \
f_0\circ j= i\circ \widehat{u}.$$ We then have $\overline{f_0^{-1}\circ \alpha_g \circ f_0} = \overline{\beta_{u(g)}}$ and $(f_0^{-1}\circ \alpha_g \circ f_0)_{|i(\widehat{\Gamma})}={\rm id}=(\beta_{u(g)})_{|i(\widehat{\Gamma})}$. The condition  $\Hom(H,\widehat{\Gamma})=\{1\}$ and Lemma \ref{lemm:autocentral} then ensure that $f_0^{-1}\circ \alpha_g \circ f_0 = \beta_{u(g)}$, and we conclude that the cocentral actions $(i,\alpha)$ and $(j,\beta)$ are equivalent.
	\end{proof}

\begin{example}
Let $p\geq 3$ be a prime number. There are exactly two non-isomorphic non-trivial graded twistings of $\Oo(\SL_2(\mathbb F_p))$. The details will be given in Section 5.
\end{example}	

The previous theorem has the following very convenient consequence when $\Gamma=\mathbb Z_2$.

\begin{theorem}\label{thm:firstmainbis}
	Let $G$ be a finite group with cyclic center, let $(i,\alpha)$ and $(j,\beta)$ be cocentral actions of $\mathbb Z_2$ on $G$, and put $H=G/i(\widehat{\mathbb Z_2})=G/j(\widehat{\mathbb Z_2})$. Assume that $H^2(H,k^\times)$ is cyclic and that $\Hom(H,\mathbb Z_2)=\{1\}$. Then the following assertions are equivalent.
	\begin{enumerate}
		\item The Hopf algebras $\Oo(G)^{i,\alpha}$ and $\Oo(G)^{j,\beta}$ are isomorphic.
		\item The cocentral actions $(i,\alpha)$ and $(j,\beta)$ are equivalent.
		\item The cocentral actions $(i,\alpha)$ and $(j,\beta)$ are weakly equivalent.
	\end{enumerate}
\end{theorem}

\begin{proof}
 The universal coefficient theorem provides the following exact sequence
\[
 0 \to {\rm Ext}^1(H_1(H), \mathbb Z_2) \to H^2(H, \mathbb Z_2) \to {\rm Hom}(H_2(H),\mathbb Z_2) \to 0
\]
The assumption $\Hom(H,\mathbb Z_2)=\{1\}$ implies that $\widehat{H}\simeq H_1(H)$ has odd order, so the group on the left vanishes. Moreover $H_2(H)\simeq H^{2}(H,k^\times)$ (again by the universal coefficient theorem), so the cyclicity of  $H^2(H,k^\times)$ yields that   $|H^2(H, \mathbb Z_2)|\leq 2$, and we can apply Theorem \ref{thm:firstmain}.
\end{proof}

\begin{rem}
 It is natural to wonder whether Theorem \ref{thm:firstmain} pertinently applies outside the case $m=2$. There is, at least, the example $G=H\times \mathbb Z_m$ where $H$ is a group with $\widehat{H}=\{1\}$ and $|H^2(H,\mathbb Z_m)|\leq 2$, and if $H^2(H,\mathbb Z_m)\simeq \mathbb Z_2$ (which, by the universal coefficient theorem, will hold if $H^2(H,k^\times)\simeq \mathbb Z_2$ and $m$ is even), the group $G$ obtained as the non split central extension $1\to \mathbb Z_m \to G\to H\to 1$ corresponding to the non trivial cohomology class.
\end{rem}


 

\section{Abelian cocentral  extensions of Hopf algebras}

To go beyond Theorem \ref{thm:firstmain}, it will be convenient to work in the more general framework of abelian cocentral extensions. As already said in the introduction, this is a very well studied and understood framework \cite{ad,mas,mas95,nat,kas00,kro} (even in more general situations, dropping the cocentrality assumption), but we propose a detailed exposition of the structure of Hopf algebras fitting into abelian cocentral extensions, both for the sake of self-completeness and of introducing the appropriate notations, and also because we think that some of our formulations have some interest.

\subsection{Generalities}
We recall first the concept and the structure of the Hopf algebras arising from abelian cocentral extensions.
There is a general notion of exact sequence of Hopf algebras \cite{ad}, but in this paper we will only need the cocentral ones.  

\begin{definition}
 A  sequence  of Hopf algebra maps
\begin{equation*}k \to B \overset{i}\to A \overset{p}\to L \to
k\end{equation*} is said to be cocentral exact if $i$ is injective, $p$ is surjective and cocentral, $p\circ i=\varepsilon(-)1$ and  $i(B) = A^{{\rm co} p} = \{ a \in A:\, (\id \otimes p)\circ\Delta(a) = a \otimes 1. \}$. When $B$ is commutative, a cocentral exact sequence as above is called an \textsl{abelian cocentral extension}. 
\end{definition}

\begin{example}
 Let $(i,\alpha)$ be a cocentral action of a group $\Gamma$ on a linear algebraic group $G$. Then 
 \begin{equation*}k \to \Oo(G/i(\widehat{\Gamma})) \to \Oo(G) \to k\Gamma \to
k\end{equation*}
is cocentral abelian extension, as well as
 \begin{equation*}k \to \Oo(G/i(\widehat{\Gamma})) \to \Oo(G)^{i,\alpha} \to k\Gamma \to
k.\end{equation*}
Hence graded twists of function algebras fit into appropriate abelian cocentral  extensions.
\end{example}

We now restrict ourselves to finite dimensional Hopf algebras. In this case the abelian cocentral extensions are of the form 
\begin{equation*}k \to \Oo(H) \to A \to k\Gamma \to
k\end{equation*}
for some finite groups $H,\Gamma$. There are some general descriptions of the Hopf algebras $A$ fitting into such abelian cocentral extensions using various actions and cocycles (see \cite{ad,mas}). Since we only will consider the case when $\Gamma$ is cyclic, there is an even simpler description, inspired by \cite{mas95}, that we give now. We start with a lemma.

\begin{lemma}\label{lemm:alga(h)}
 Let $H$ a finite group, let $\theta  \in \Aut(H)$ with $\theta^m={\rm id}_H$ for some $m\geq 1$, and let $a : H\to k^\times$. Consider the algebra $A_m(H,\theta,a)$ defined by the quotient of the free product algebra 
$\Oo(H)*k[g]$ by the relations :
$$g^m=a,  ~\  ~\ \  ge_x = e_{\theta(x)}g,  \forall x \in H.$$
Then the set $\{e_xg^i, \ x \in H, \ 0\leq i\leq m-1\}$ linearly spans  $A_m(H,\theta,a)$, and is a basis if and only if $a\circ \theta=a$.
\end{lemma}
 
\begin{proof}
 It is clear from the defining relations that $\{e_xg^i, \ x \in H, \ 0\leq i\leq m-1\}$ linearly spans  $A_m(H,\theta,a)$. The defining relations give that for any $\phi\in \Oo(H)$, we have $g\phi=(\phi\circ\theta^{-1})g$, and since $a=g^m$ must be central, we see from this that if the above set is linearly independent, we have $a\circ\theta =a$.  

To prove the converse, we recall a general construction. Let $R$ be a commutative algebra endowed with an
action of a group $\Gamma$, $\alpha : \Gamma \to \Aut(R)$,
and let $\sigma : \Gamma \times \Gamma \to R^{\times}$ be a 2-cocycle according to this action:
$$\alpha_r(\sigma(s,t))\sigma(r,st)=\sigma(rs,t)\sigma(s,t), \ \forall r,s,t \in \Gamma$$
The crossed product algebra $R\#_\sigma k\Gamma$ is then defined to be the algebra having $R\otimes k\Gamma$ as underlying vector space, and product defined by 
$$x\#r \cdot y\#s = x\alpha_r(y)\sigma(r,s)\#rs$$ 
Assume furthermore that $\Gamma=\mathbb Z_m=\langle g\rangle$ is cyclic, consider an element $a \in R^\times$ that is $\mathbb Z_m$-invariant, and define the algebra $A$ to be the quotient of the free product $R*k[X]$ by the relations $Xb= \alpha_g(b)X$ and $X^m =a$. Since $a$ is invariant under the $\mathbb Z_m$-action, the classical description of the second cohomology of a cyclic group shows that there exists a $2$-cocycle  $\sigma : \mathbb Z_m\times \mathbb Z_m \to R^{\times}$ such that $\sigma(g,g)\cdots \sigma(g,g^{m-1})=a$. From this we get an algebra map 
\begin{align*}
 A &\longrightarrow R \#_\sigma k\mathbb Z_m \\
b \in R, \ X & \longmapsto b\#1, \ 1\# g.
\end{align*}
Applying this to $R=\mathcal O(H)$, the $\mathbb Z_m$-action on it induced by $\theta$ and the assumption that $a$ is invariant yields that $\{e_xg^i, \ x \in H, \ 0\leq i\leq m-1\}$ is a linearly independent set since its image is in the crossed product algebra $\Oo(H)\#_\sigma k\mathbb Z_m$.
\end{proof}

\begin{definition}\label{def:datum}
 Let $m\geq 1$. An $m$-\textsl{datum} is a quadruple $(H,\theta,a,\tau)$ consisting of a finite group $H$, an automorphism $\theta \in {\rm Aut}(H)$ such that $\theta^m={\rm id}_H$, a map $a : H \to k^\times$ such that $a\circ \theta=a$ and $a(1)=1$, and a $2$-cocycle 
$\tau : H \times H \to k^\times$ such that for any $x,y \in H$
$$\left(\prod_{i=0}^{m-1}\tau(\theta^i(x),\theta^i(y))\right)a(x)a(y)= a(xy).$$ 
\end{definition}

We now check $m$-data as above produce Hopf algebras fitting into abelian cocentral extensions, and that any such Hopf algebra arises in this way.

\begin{proposition}\label{prop:hopfa(h)}
 Let $(H,\theta,a,\tau)$ be an $m$-datum, and consider the algebra $A_m(H,\theta,a)$ defined by the quotient of the free product algebra 
$\Oo(H)*k[g]$ by the relations :
$$ge_x = e_{\theta(x)}g, \forall x \in H, ~\ ~\ \   g^m=a.$$
\begin{enumerate}
 \item There exists a unique Hopf algebra structure on $A_m(H,\theta,a)$ extending that of $\Oo(H)$ and such that
$$\Delta(g) = \sum_{y,z\in H} \tau(y,z)e_yg\otimes e_zg, \ ~\ ~\ \varepsilon(g)=1.$$
We denote by $A_m(H,\theta, a, \tau)$ the resulting Hopf algebra.
\item The Hopf algebra $A_m(H,\theta,a,\tau)$ has dimension $m|H|$ and fits into an abelian cocentral extension
$$k \to \Oo(H) \to A_m(H,\theta,a,\tau) \overset{p} \to k\mathbb Z_m \to k$$
where $p$ is the Hopf algebra map defined by $p_{|\Oo(H)}=\varepsilon$ and $p(g)=g$ (here $g$ denotes any fixed generator of $\mathbb Z_m$).
\end{enumerate}
\end{proposition}

\begin{proof}
It is a straightforward verification, using the axioms of $m$-data, that there indeed exists a Hopf algebra structure on $A_m(H,\theta,a)$ as in the statement. That  $A_m(H,\theta,a,\tau)$ has dimension $m|H|$, follows from Lemma \ref{lemm:alga(h)}, while the last statement follows easily from the decomposition $A_m(H,\theta,a,\tau)=\oplus_{k=0}^{m-1}\Oo(H)g^k$.
\end{proof}

\begin{proposition}\label{prop:ext->a(h)}
 Let $A$ be a finite-dimensional Hopf algebra fitting into an abelian  cocentral extension $$k \to \Oo(H) \to A  \to k\mathbb Z_m \to k$$
Then there exists an $m$-datum $(H,\theta,a,\tau)$ such that $A\simeq A_m(H,\theta,a,\tau)$ as Hopf algebras.
\end{proposition}

\begin{proof}
 To simplify the notation, we will identify $\Oo(H)$ with its image in $A$, so that $A_e=\Oo(H)$. The finite-dimensionality assumption ensures that the extension is cleft (see e.g \cite[Theorem 3.5]{md} or \cite[Theorem 2.4]{sch}). Here this simply means that for any $h \in \mathbb Z_m$, there exists an invertible element $u_h$ in $A_h$, that we normalize so that $\varepsilon(u_h)=1$, and hence $p(u_h)=h$, where $p: A \to k\mathbb Z_m$ is the given cocentral surjective Hopf algebra map. We have $A_eu_h\subset A_h$ and for $b\in A_h$, we can write $b=bu_h^{-1}u_h\in A_eu_h$, hence $A_h=A_eu_{h}$. 

Fix now a generator $g$ of  $\mathbb Z_m$ and $u_g$ as above. We have $u_g^m\in A_{g^m}=A_e$, and we put $a=u_g^m$. Since $\Delta(A_g)\subset A_g\otimes A_g$ we have
$\Delta(u_g) = \sum_{x,y\in H} \tau(x,y)e_xu_g\otimes e_yu_g$ for scalars $\tau(x,y)\in k$, these scalars all being non-zero since $\Delta(u_g)$ is invertible. The coassociativity and counit conditions give that the map $\tau : H\times H \to k^\times$ defined in this way is a $2$-cocycle. We have $u_gA_eu_{g}^{-1} \subset A_e$ and hence we get an automorphism $\alpha:={\rm ad}(u_g)$ of the algebra $A_e$, satisfying $\alpha^m={\rm id}$ since $u_g^m\in A_e$ and $A_e$ is commutative. It is a direct verification to check that $\alpha$ is as well a coalgebra automorphism, and hence a Hopf algebra automorphism of $A_e=\Oo(H)$, necessarily arising from an automorphism $\theta$ of $H$, with $\alpha(\phi)=\phi\circ\theta^{-1}$ for $\phi\in \Oo(H)$. Clearly $\alpha(a)=a$, $\varepsilon(a)=1$, and one checks that the last condition defining an $m$-datum is fulfilled by comparing $\Delta(u_g)^m$ and $\Delta(a)$. We thus obtain an $m$-datum $(H,\theta,a,\tau)$ and it is straightforward to check that there exists a Hopf algebra map $ A_m(H,\theta,a,\tau)\to A$, $\phi\in\Oo(H) \mapsto \phi$, $g \mapsto u_g$. Combining Lemma \ref{lemm:alga(h)} and the first paragraph in the proof, we see that this is an isomorphism.
\end{proof}

\subsection{Equivalence of $m$-data and the isomorphism problem}
The main question then is to classify the Hopf algebras $A_m(H,\theta,a,\tau)$ up to isomorphism.
For this, the following equivalence relation on $m$-data will arise naturally.

\begin{definition}\label{def:equivdata}
	Two $m$-data $(H,\theta,a,\tau)$ and $(H',\theta',a', \tau')$ are said to be equivalent if   there exists a group isomorphism $f : H \to H'$, a map $\varphi : H' \to k^{\times}$ with $\varphi(1)=1$ and $l\in\{1, \ldots ,m-1\}$ prime to $m$ such that the following conditions hold, for any $x,y \in H'$:
	\begin{enumerate}
		\item $\theta'^l=f\circ \theta \circ f^{-1}$ ;
		\item $\left(\prod_{k=0}^{m-1}\varphi(\theta'^{k}(y))\right)a'(y)^l=a(f ^{-1}(y))$ ;
		\item $\left(\prod_{k=0}^{l-1}\tau'(\theta'^{-k}(x),\theta'^{-k}(y))\right)\varphi(xy)=\tau(f^{-1}(x),f^{-1}(y))\varphi(x)\varphi(y)$.
	\end{enumerate}
\end{definition}

It is not completely obvious that the above relation is an equivalence relation, but this follows from
the following basic result,  which is a partial answer for the classification problem of the Hopf algebras $A_m(H,\theta, a, \tau)$.

\begin{proposition}\label{prop:isoext}
 Let $(H,\theta,a,\tau)$ and $(H',\theta',a', \tau')$ be $m$-data. The following assertions are equivalent.
\begin{enumerate}
 \item There exists a Hopf algebra isomorphism $F : A_m(H,\theta,a,\tau) \to A_m(H',\theta',a',\tau')$ and a group automorphism $u\in \Aut(\mathbb Z_m)$ making the following diagram commute:
$$\xymatrix{
A_m(H,\theta,a,\tau) \ar[r]^-p \ar[d]^F  & k \mathbb Z_m  \ar[d]^u  \\
A_m(H',\theta',a',\tau') \ar[r]^-{p'} & k \mathbb Z_m   
}$$

\item The $m$-data $(H,\theta,a,\tau)$ and $(H',\theta',a', \tau')$ are  equivalent.
\end{enumerate}
 
\end{proposition}

\begin{proof}
 Assume that $F$ and $u$ as above are given, and put $A=A_m(H,\theta,a,\tau)$ and $B=A_m(H',\theta',a',\tau')$. The commutativity of the diagram yields, at the level of gradings, that $F(A_h)=B_{u(h)}$ for any $h \in \mathbb Z_m$. Then $F$ induces an isomorphism $A_e=\Oo(H) \to \Oo(H')$ coming from a group isomorphism $f : H \to H'$ such that $F(\phi)=\phi\circ f^{-1}$ for any $\phi \in \Oo(H)$. Pick a generator $g$ of $\mathbb Z_m$. We have $F(A_g)= B_{u(g)}=B_{g^l}$ for some $l\in\{1, \ldots ,m-1\}$  prime to $m$. Since $B_{g^l}=B_eg^l$, there exists $\varphi\in \Oo(H)^\times$ such that $F(g)=\varphi g^l$. The fact that $F$ is a coalgebra map yields that $\varphi(1)=1$ and relations (3). The compatibility of the algebra map $F$ with the relations $ge_x=e_{\theta(x)}g$ yields relation (1), while compatibility with the relation $g^m=a$ yields relation (2).
 
 Conversely, given $f$, $l$ and $\varphi$ as in Definition \ref{def:equivdata}, it is a direct verification to check that there exist Hopf algebra isomorphism $F : A_m(H,\theta,a,\tau) \to A_m(H',\theta',a',\tau')$ defined by $F(e_x)=e_{f(x)}$ and $F(g)=\varphi g^l$, and satisfying $u\circ p=p'\circ F$ for $u$ given by $u(g)=g^l$.
\end{proof}


\begin{corollary}\label{cor:basiciso}
 Let  $(H,\theta,a,\tau)$ be an $m$-datum. 
 \begin{enumerate}
 	\item Let $f\in \Aut(H)$ and let $l\geq 1$ be prime to $m$. Then $(H,f\circ\theta^l\circ f^{-1},a\circ f^{-1},\tau')$, with $\tau'=\prod_{k=0}^{l-1}\tau\circ \theta^k \times \theta^k \circ f^{-1}\times f^{-1}$, is an $m$-datum and $$A_m(H,\theta,a,\tau)\simeq A_m\left(H,f\circ \theta^l\circ f^{-1},(a\circ f^{-1})^l,\prod_{k=0}^{l-1}\tau\circ \theta^k \times \theta^k \circ f^{-1}\times f^{-1}\right)$$
 	 as Hopf algebras.
 	 \item Let $\tau'\in Z^2(H,k^\times)$ be cohomologous to $\tau$. There exists $a' : H \to k^\times$ such that 
 	 $(H, \theta,a',\tau')$ is a datum and   
 	 $$A_m(H,\theta,a,\tau)\simeq A_m(H, \theta,a',\tau')$$
 	 as Hopf algebras.
 \end{enumerate}
In particular,  if $\theta_1, \ldots,\theta_r$ is  a set of representative of the conjugacy classes of elements whose order divides $m$ in $\Aut(H)$, and  if $\tau_1,\ldots , \tau_s$ is a set of representative of $2$-cocycles in $H^2(H,k^\times)$, there exist $i\in \{1,\ldots , r\}$, $j\in \{1,\ldots,s\}$ and $a': H \to k^\times$ such that $(H,\theta_i,a',\tau_j)$ is a datum and $A_m(H,\theta,a, \tau)\simeq
 A_m(H,\theta_i,a', \tau_j)$.
\end{corollary}

\begin{proof}
The first assertion is easily obtained via the previous proposition. For the second one, let  $\mu :  H\to k^\times$ be such that $\tau'=\tau\partial(\mu)$. The result is again a direct consequence of the previous proposition, taking $a'=a\left( \prod_{i=0}^{m-1}\mu\circ\theta^{i}\right)^{-1}$.
The final assertion is then easily obtained by combining (1) and (2).
\end{proof}

\begin{rem}\label{rem:normalized}
 Let $(H,\theta, a, \tau)$ be an $m$-datum. Since $a\circ \theta = a$, there exists a map $\mu : H \to k^\times$ such that
$\mu \circ\theta =\mu$ and $\mu^m=a$. The cocycle $\tau'= \tau \partial(\mu)$ then satisfies 
$$\prod_{k=0}^{m-1}\tau'\circ \theta^k \times \theta^k=1.$$ 
Hence, by Corollary \ref{cor:basiciso}, the $m$-datum $(H,\theta, a, \tau)$ is equivalent to an $m$-datum $(H,\theta, a', \tau')$ with $a'\in \widehat{H}$. Such a datum with $a'\in \widehat{H}$ will be said to be \textsl{normalized}. It is therefore tempting to work only with normalized data, but this forces to change the cocycle for each choice of automorphism $\theta$, and can be inconvenient in practice if we have ``nice'' representatives for $2$-cocycles over $H$. We will therefore work with the general notion of an $m$-datum, as given in Definition \ref{def:datum}.
\end{rem}

\begin{rem}\label{rem:opext}
Fix $\theta\in \Aut(H)$ with $\theta^m={\rm id}$. Kac's group  ${\rm Opext}_\theta(k\mathbb Z_m, \mathcal O(H))$ \cite{kac} can be described as the set of pairs $(a,\tau)\in \widehat{H}\times Z^2(H,k^\times)$ such that $(H,\theta, a,\tau)$ is a normalized $m$-datum modulo the equivalence relation defined by
$(a,\tau)\sim (a',\tau')\iff \exists \varphi : H\to k^\times$ with $\varphi\cdot \varphi\circ\theta \in \widehat{H}$,  $\left(\prod_{k=0}^{m-1}\varphi\circ \theta^{k}\right)a'=a$ and  $\tau'=\tau\partial(\varphi)$. The group law is by the ordinary multiplication on the components. The group ${\rm Opext}_\theta(k\mathbb Z_m, \mathcal O(H))$ is known to be possibly difficult to compute (see \cite{mas97}, and \cite{gm} for a recent contribution), hence the problem of the  description of $m$-data up to equivalence is a fortiori a non-obvious one as well.
\end{rem}

Proposition \ref{prop:isoext} is  in general 
 not sufficient to classify the Hopf algebras $A_m(H,\theta,a,\tau)$ up to isomorphism. However, in the context of Lemma \ref{lemm:semi-univ}, it can be sufficient. Thus we need to analyse furthermore the Hopf algebras $A_m(H,\theta,a,\tau)$ to determine when Lemma \ref{lemm:semi-univ} is applicable. For this we introduce a number of groups associated to an $m$-datum.
 
 \begin{definition}
 	 Let  $(H,\theta,a,\tau)$ be an $m$-datum. 
 	 \begin{enumerate}
 	 	\item We put $Z_{\tau, \theta}(H)=\{x \in Z(H) \ | \ \tau(\theta^i(x),y)=\tau(y,\theta^i(x)), \ \forall y\in H, \ \forall i,  \ 0\leq i \leq m-1\}$. This a central subgroup of $H$, and we get, by restriction, a new $m$-datum  $(Z_{\tau, \theta}(H),\theta,a,\tau)$.
 	 	\item Let $H^\theta$ be the subgroup of $H$ formed by elements that are invariant under $\theta$. The group $G(H,\theta,a,\tau)$ is the group whose elements are pairs $(x,\lambda)\in H^\theta \times k^\times$ satisfying $\lambda^m=a(x)$, and whose group law is defined by $(x,\lambda)\cdot (y,\mu)= (xy,\tau(x,y)\lambda\mu)$. 
 	 	\item We denote by $G_0(H,\theta,a,\tau)$ the group $G(Z_{\tau,\theta}(H),\theta,a,\tau)$, thus consisting of pairs $(x,\lambda)\in Z(H)^\theta \times k^\times$ with $x$ satisfying $\lambda^m=a(x)$ and $\tau(x,y)=\tau(y,x)$,  $\forall y\in H$.
 	 \end{enumerate}
 	\end{definition}
 
 It is easy to check that $G(H,\theta,a,\tau)$ is indeed a group, fitting into a central exact sequence
 $$1 \to \mu_m \to G(H,\theta,a,\tau) \to H^\theta\to 1.$$

\begin{proposition}\label{prop:univa(h)}
	Let $(H,\theta,a,\tau)$ be an $m$-datum. We have a universal cocentral exact sequence
	$$k \to \Oo(H/Z_{\tau, \theta}(H)) \to A_m(H,\theta,a,\tau) \to A_m(Z_{\tau, \theta}(H),\theta,a,\tau)\to k.$$
\end{proposition}

\begin{proof}
It is easily seen that there is a surjective Hopf algebra map $$p: A(H,\theta,a,\tau) \to A(Z_{\tau, \theta}(H),\theta,a,\tau)$$ with $p(g)=g$ and such that for $\phi \in \Oo(H)$, $p(\phi)$ is the restriction of $\phi$ to $Z_{\tau, \theta}(H)$. The cocentrality of $p$ follows from the centrality of the group $Z_{\tau, \theta}(H)$ in $H$, and it is easy to see that $p$ induces the announced cocentral exact sequence. We thus have to prove the universality of $p$. For this consider a cocentral Hopf algebra map $q: A(H,\theta,a,\tau) \to B$.  The cocentrality of $q$ yields that $q(e_x)=0$ if $x\not \in Z(H)$, and that for any $x \in Z(H)$ and $y \in H$, $\tau(x,y)q(e_x)q(g)=\tau(y,x)q(e_x)q(g)$. Hence  $\tau(x,y)q(e_x)=\tau(y,x)q(e_x)$ and $\tau(x,y)=\tau(y,x)$ if $q(e_x)\not=0$. Let $T:=\{x \in H \ | \ q(e_x)\not=0\}$. Since $q(g)q(e_x)q(g)^{-1}= q(ge_xg^{-1})=q(e_{\theta(x)})$ we thus see that $T\subset  Z_{\tau, \theta}(H)$. We then easily check that there exists a Hopf algebra map $f : A(Z_{\tau, \theta}(H),\theta,a,\tau)\to B$ with $f(e_x)=q(e_x)$ and $f(g)=q(g)$, as needed.
	\end{proof}

We now proceed to analyse the structure of the Hopf algebras $A_m(H,\theta,a,\tau)$, with first the following basic result.

\begin{proposition}\label{prop:a(h)co}
 Let $(H,\theta,a,\tau)$ be an $m$-datum. 
\begin{enumerate}
 \item The Hopf algebra $A_m(H,\theta,a,\tau)$ is commutative if and only if $\theta={\rm id}_H$. More generally, the abelianisation of $A_m(H,\theta,a,\tau)$ is the Hopf algebra $\Oo(G(H,\theta,a,\tau))$.
\item The Hopf algebra $A_m(H,\theta,a,\tau)$ is cocommutative if and only if $H$ is abelian and $\tau$ is symmetric, i.e. $\tau(x,y)=\tau(y,x)$ for any $x,y \in H$.
\end{enumerate}

\end{proposition}

\begin{proof}
The assertions regarding the commutativity or cocommutativity of $A_m(H,\theta,a,\tau)$ are easily seen using Lemma \ref{lemm:alga(h)}.  An algebra map $\chi : A_m(H,\theta,a,\tau)\to k$ corresponds to a pair $(x,\lambda)\in H\times k^\times$, with $\chi(\phi)=\phi(x)$ for any $\phi \in \Oo(H)$ and $\phi(a)=\lambda$. The compatibility of $\chi$ with the defining relations of $A_m(H,\theta,a,\tau)$ is easily seen to be equivalent to the condition that $(x,\lambda)\in G(H,\theta,a,\tau)$, and an immediate calculation shows that the group law in ${\rm Alg}(A_m(H,\theta,a,\tau),k)$ corresponds to the group law in $G(H,\theta,a,\tau)$. Thus the abelianization of $A_m(H,\theta,a,\tau)$, which is the algebra of functions on ${\rm Alg}(A_m(H,\theta,a,\tau),k)$, is isomorphic to $\Oo(G(H,\theta,a,\tau))$.
\end{proof}

We now discuss when the universal grading group of  $A_m(H,\theta,a,\tau)$ is cyclic.

\begin{proposition} \label{prop:a(h)cyclic} Let $(H,\theta,a,\tau)$ be an $m$-datum. 
	\begin{enumerate}
 \item The Hopf algebra $A_m(H,\theta,a,\tau)$ has a universal cyclic grading group if and only if the group $G_0(H,\theta,a,\tau)$ is cyclic and the restriction of $\theta$ to $Z_{\tau, \theta}(H)$ is trivial.
 \item The natural cocentral Hopf algebra map $p: A_m(H,\theta,a,\tau)\to k\mathbb Z_m$ is universal if and only if the group $Z_{\tau, \theta}(H)$ is  trivial.
\end{enumerate}
\end{proposition}

\begin{proof}
(1) Assume that	 $A_m(H,\theta,a,\tau)$ has a universal cyclic grading group. By Proposition \ref{prop:univa(h)}, we have that $A_m(Z_{\tau, \theta}(H),\theta,a,\tau)$ is the group algebra of a cyclic group, and in particular is commutative. Then by (1) in Proposition \ref{prop:a(h)co}, the restriction of $\theta$ to $Z_{\tau, \theta}(H)$ is trivial and  $G_0(H,\theta,a,\tau)= G(Z_{\tau,\theta}(H),\theta,a,\tau)$ is cyclic.

Conversely, if the restriction of $\theta$ to $Z_{\tau, \theta}(H)$ is trivial, then by (1) in Proposition \ref{prop:a(h)co}, the Hopf algebra $A_m(Z_{\tau, \theta}(H),\theta,a,\tau)$ is commutative and isomorphic to $\Oo(G(Z_{\tau,\theta}(H),\theta,a,\tau))$. Assuming moreover that $G_0(H,\theta,a,\tau)= G(Z_{\tau,\theta}(H),\theta,a,\tau)$ is cyclic, we obtain that $A_m(Z_{\tau, \theta}(H),\theta,a,\tau)$ is the group algebra of a cyclic group, and we conclude by Proposition \ref{prop:univa(h)}.

(2)  The canonical surjection $A_m(Z_{\tau,\theta}(H),\theta,a,\tau) \to k \mathbb Z_m$ is an isomorphism if and only if $Z_{\tau, \theta}(H)$ is  trivial, because $\dim(A_m(Z_{\tau,\theta}(H),\theta,a,\tau))=m|Z_{\tau,\theta}(H)|$. Hence Proposition \ref{prop:univa(h)} yields the result.
\end{proof}	

The previous result leads us to introduce some more vocabulary.

\begin{definition}
	An $m$-datum $(H,\theta,a,\tau)$ is said to be \textsl{cyclic} (resp. \textsl{reduced}) if the group $G_0(H,\theta,a,\tau)$ is cyclic and the restriction of $\theta$ to $Z_{\tau, \theta}(H)$ is trivial (resp. if the group $Z_{\tau, \theta}(H)$ is  trivial).
\end{definition}

We get our most useful result for the classification of Hopf algebras of type $A_m(H,\theta,a,\tau)$. 

\begin{proposition}\label{prop:isoa(h)}
	Let $(H,\theta,a,\tau)$ and $(H',\theta',a', \tau')$ be cyclic $m$-data. The following assertions are equivalent.
	\begin{enumerate}
		\item The Hopf algebras $A_m(H,\theta,a,\tau)$  and $A_m(H',\theta',a',\tau')$ are isomorphic.
		
		\item The data $(H,\theta,a,\tau)$ and $(H',\theta',a', \tau')$ are  equivalent.
		\end{enumerate}
\end{proposition}

\begin{proof}
We have (2)$\Rightarrow$(1) by Proposition 	\ref{prop:isoext}. Assuming that (1) holds, Proposition \ref{prop:a(h)cyclic} ensures that we are in the situation of Lemma  \ref{lemm:semi-univ}, which in turn ensures that we are in the situation of (1) in Proposition \ref{prop:isoext}, so that (2) holds.
\end{proof}

Combining Propositions \ref{prop:ext->a(h)} and \ref{prop:isoa(h)}, we finally obtain the main result of the section.

\begin{theorem}\label{thm:isoext}
	Let $H$ be a finite group and let $m\geq 1$.  The map $(H,\theta,a, \tau) \mapsto A_m(H,\theta, a , \tau)$ induces a bijection between the following sets:
	\begin{enumerate}
		\item equivalence classes of cyclic (resp. reduced) $m$-data having $H$ as underlying group;
		\item isomorphism classes of Hopf algebras $A$ fitting into an abelian cocentral extension
		$$k \to \Oo(H) \to A \to k\mathbb Z_m \to k$$ and having a cyclic universal grading group (resp.~ having $\mathbb Z_m$ as universal grading group).
	\end{enumerate}
\end{theorem}

\begin{corollary}\label{cor:isoextz(h)}
Let $H$ be a finite group with $Z(H)=\{1\}$ and let $m\geq 2$.  The map $(H,\theta,a, \tau) \mapsto A_m(H,\theta, a , \tau)$ induces a bijection between the following sets:
	\begin{enumerate}
		\item equivalence classes of  $m$-data having $H$ as underlying group;
		\item isomorphism classes of Hopf algebras $A$ fitting into an abelian cocentral extension
		$$k \to \Oo(H) \to A \to k\mathbb Z_m \to k.$$ 
	\end{enumerate}
 
\end{corollary}

\begin{proof}
This follows from the previous theorem, since the assumption $Z(H)=\{1\}$ ensures that all the $m$-data  $(H,\theta,a, \tau)$ are reduced and that all the corresponding abelian cocentral extensions
		$k \to \Oo(H) \to A \to k\mathbb Z_m \to k$ are universal. 
\end{proof}

\subsection{Classification results}\label{subsecClass} We now apply Theorem \ref{thm:isoext} and Corollary \ref{cor:isoextz(h)} to obtain effective classification results for Hopf algebras fitting into abelian cocentral extensions, under various assumptions.

The set of equivalence classes of $m$-data has a very simple description under some strong  assumptions on $H$, and then the previous result takes the following simple form, where we use the following notation: if $G$ is a group and $m\geq 1$, the set ${\rm CC}_m^\bullet(G)$ is the set of  elements of $G$ such that $x^m=1$ and $x\not=1$, modulo the equivalence relation defined by $x\sim y\iff$ there exists $l$ prime to $m$ such that $x^l$ is conjugate to $y$.  When $m=2$, ${\rm CC}_2^\bullet(G)$ is just the set of conjugacy classes of elements of order $2$ in $G$.



\begin{theorem}\label{cor:isoext2}
 Let $H$ be a finite group with $\widehat{H}=\{1\}=Z(H)$ and $H^{2}(H,k^\times)\simeq \mathbb Z_2$. Then for any $m\geq 2$, there is a bijection between the set of isomorphism classes of noncommutative Hopf algebras $A$ fitting into an abelian cocentral extension
		$$k \to \Oo(H) \to A \to k\mathbb Z_m \to k$$
and  
\begin{enumerate} 
	\item if $m$ is odd, the set ${\rm CC}_m^\bullet(\Aut(H))$;
	\item if $m$ is even, the set  ${\rm CC}_m^\bullet(\Aut(H))\times H^2(H,k^\times)$.
\end{enumerate}
\end{theorem}

\begin{proof}
Since $Z(H)=\{1\}$, the previous corollary ensures that we have a bijection between the set of isomorphism classes of noncommutative Hopf algebras as above and the set of equivalence classes of $m$-data $(H,\theta, a , \tau)$ with  $\theta\not={\rm id}$.

The key point, to be used freely, is that, since $H^{2}(H,k^\times)\simeq \mathbb Z_2$, for any $\theta\in \Aut(H)$ and $\tau\in Z^2(H,k^\times)$, we have $[\tau]=[\tau\circ \theta\times \theta]$ and $[\tau][\tau\circ \theta\times \theta]=1$ in $H^2(H,k^\times)$.

First assume that $m$ is odd. Let $(H,\theta, a , \tau)$ be an $m$-data with $\theta\not={\rm id}$. Then 
$[\tau]^m=1$ and $[\tau]=1$ since $m$ is odd, so $(H,\theta, a , \tau)$ is equivalent to some $m$-datum
$(H,\theta, a' , 1)$ with $a'=1$ since $\widehat{H}=\{1\}$. The result is then clear.

Assume now that $m$ is even,
and start with a pair $(\theta,\tau)$ where $\theta\in \Aut(H)$ satisfies $\theta^m={\rm id}$ $\theta\not={\rm id}$, and $\tau\in Z^2(H,k^\times)$. The assumption $H^{2}(H,k^\times)\simeq \mathbb Z_2$ implies again that there exists $a: H \to k^\times$ such that $\prod_{k=0}^{m-1}\tau\circ\theta^k \times \theta^k =\partial(a^{-1})$. The assumption $\widehat{H}=\{1\}$ implies that such a map $a$ is unique and satisfies $a\circ \theta = a$, so to $(\theta,\tau)$ we can unambiguously associate an $m$-datum $(H,\theta, a,\tau)$.

Consider now another such pair $(\theta',\tau')$ with $a'$ the corresponding map making  $(H,\theta', a',\tau')$ an $m$-datum. If the $m$-data $(H,\theta, a,\tau)$ and $(H,\theta', a',\tau')$ are equivalent, then there is $l$ prime to $m$ (hence $l$ is odd) such that $\theta'^l$ is conjugate to $\theta$ and
$[\tau]= [\tau']^l=[\tau']$ (remark at the beginning of the proof).

Conversely if $\theta=f\circ \theta^l \circ f^{-1}$, for $f\in \Aut(H)$ and $l$ prime to $m$, then we have, by Corollary \ref{cor:basiciso}
\begin{align*} 
(H,\theta,a,\tau) & \sim (H, f\circ \theta^l \circ f^{-1},(a\circ f^{-1})^l, \prod_{k=0}^{l-1}\tau\circ \theta^kf^{-1}\times \theta^kf^{-1}) \\
& \sim (H,\theta', (a\circ f^{-1})^l, \prod_{k=0}^{l-1}\tau\circ \theta^kf^{-1}\times \theta^kf^{-1}). \\
\end{align*}
The cocycle on the right is cohomologous to $\tau^l$, hence to $\tau$,
 and if we assume that $\tau'$ is cohomologous to $\tau$, we have 
(again thanks to Corollary \ref{cor:basiciso}) 
$$	(H,\theta, a, \tau)\sim  (H,\theta', b, \tau)\sim  (H,\theta', c, \tau')$$
for some maps $b,c$, with necessarily $c=a'$ by the discussion at the beginning of the proof. This concludes the proof.
\end{proof}

Another useful consequence of Theorem \ref{thm:isoext} is the following one, again under strong assumptions.

\begin{theorem}\label{cor:isoext3}
  Let $H$ be a finite group with $|\widehat{H}|\leq 2$, and $Z(H)=\{1\}=H^{2}(H,k^\times)$. Then for $m \geq 1$,  there is a bijection between the set of isomorphism classes of noncommutative Hopf algebras $A$ fitting into an abelian cocentral extension
		$$k \to \Oo(H) \to A \to k\mathbb Z_m \to k$$
and 
\begin{enumerate}
	\item if $m$ is odd, the set ${\rm CC}_m^\bullet({\rm Aut}(H))$;
	\item if $m$ is even, the set ${\rm CC}_m^\bullet({\rm Aut}(H))\times \widehat{H}$.
\end{enumerate}
\end{theorem}

\begin{proof}
    Corollary \ref{cor:isoextz(h)} ensures that we have a bijection between the set of isomorphism classes of noncommutative Hopf algebras as above and the set of equivalence classes of $m$-data $(H,\theta, a , \tau)$ with $\theta^m\not={\rm id}$. Then, since $H^{2}(H,k^\times)=\{1\}$, Corollary \ref{cor:basiciso} ensures that all such data are equivalent to data of type $(H,\theta, a, 1)$ (hence with $a\in \widehat{H}$). Now using that $|\widehat{H}|\leq 2$, so that ${\rm Aut}(H)$ acts trivially on $\widehat{H}$, we see that two $m$-data $(H,\theta, a, 1)$ and $(H,\theta', a', 1)$ are equivalent if and only if there exists $f\in {\rm Aut}(H)$, $\varphi \in \widehat{H}$ and $l$ prime to $m$ such that 
$$\theta'^{l}=f\circ \theta \circ f^{-1}, \ ~\ ~\ \varphi^ma'^l=a.$$
If $m$ is even, we have $\varphi^m=1$ and the last condition amounts to $a'=a$ ($l$ being then necessarily odd), again since $|\widehat{H}|\leq 2$. If $m$ is odd, we have $\varphi^m=\varphi$, and such a $\varphi$ always exists if $l$ does. This concludes the proof.
\end{proof}

To prove our next classification result, we will use the following lemma.

\begin{lemma}\label{lem:Sn}
Let $H$  be a finite group in which any automorphism is inner and such that $Z(H)=\{1\}$ and $|\widehat{H}|\leq 2$. If $(H,\theta,a,\tau)$ and $(H,\theta, a', \tau)$ are equivalent  $2$-data, then $a=a'$.
\end{lemma}
 
\begin{proof}
We first assume that our data are normalized: $\tau \cdot\tau\circ \theta \times \theta =1$ (and $a, a'\in \widehat{H}$). 
 Let $f\in{\rm Aut}(H)$ and $\varphi : H\to k^\times$ be such that 
$$f\circ\theta=\theta \circ f, \ \varphi\cdot\varphi\circ \theta \cdot a'=a\circ f^{-1}, \ \tau = \partial(\varphi)\tau\circ f^{-1}\times f^{-1}.$$
Writing $\theta={\rm ad}(x)$ and $f^{-1}={\rm ad}(y)$, we then have $xy=yx$ since $Z(H)=\{1\}$ and 
$$\varphi\cdot\varphi\circ \theta \cdot a'=a\circ f^{-1}, \ \tau = \partial(\varphi)\tau\circ f^{-1}\times f^{-1}=\partial(\varphi)\partial(\mu_y^{-1})\tau$$
where $\mu_y$ is as in Lemma \ref{lem:tauinner}. Hence $\varphi = \chi \mu_y$ for some $\chi\in \widehat{H}$, and 
$$\varphi \cdot \varphi\circ \theta = \chi\cdot \chi\circ \theta \cdot \mu_y\cdot \mu_y\circ\theta.$$ 
Since $|\widehat{H}|\leq 2$ and $\theta$ is inner, we obtain $\varphi \cdot \varphi\circ \theta =  \mu_y\cdot \mu_y\circ\theta= \mu_y\cdot \mu_y\circ{\rm ad}(x)$. For $z\in H$, we have
\begin{align*}
\mu_y\circ {\rm ad}(x)(z) &=\tau(yxzx^{-1},y^{-1})\tau(y,xzx^{-1})\tau(y,y^{-1})^{-1} \\
& = \tau(xyzx^{-1},xy^{-1}x^{-1})\tau(xyx^{-1},xzx^{-1})\tau(xyx^{-1},xy^{-1}x^{-1})^{-1}\\
& = \tau(yz,y^{-1})^{-1}\tau(y,z)^{-1}\tau(y,y^{-1}) \\
& = \mu_y(z)^{-1}
\end{align*}
where we have used the fact that our datum is normalized and that $xy=yx$. Hence $\varphi\cdot\varphi\circ \theta=1$, and $a=a'$. 

In general, recall (See remark \ref{rem:normalized}) that $(H,\theta,a,\tau)$ and $(H,\theta,a',\tau')$ are
respectively equivalent to normalized $2$-data $(H,\theta,b,\tau')$ and $(H,\theta,b',\tau')$, hence 
$b=b'$ from the normalized case, and $a=a'$ by the construction of $b$ and $b'$ from $a$ and $a'$ (see the proof of Corollary \ref{cor:basiciso}).
\end{proof}

\begin{theorem}\label{cor:isoext4}
  Let $H$ be a finite group in which any automorphism is inner and with $|\widehat{H}|\leq 2$, $Z(H)=\{1\}$ and $|H^{2}(H,k^\times)|\leq 2$. Then  there is a bijection between the set of isomorphism classes of noncommutative Hopf algebras $A$ fitting into an abelian cocentral extension
		$$k \to \Oo(H) \to A \to k\mathbb Z_2 \to k$$
and  the set ${\rm CC}_2^\bullet(H) \times\widehat{H}\times H^{2}(H,k^\times)$.
\end{theorem}

\begin{proof}
 As before, in view of the assumption $Z(H)=\{1\}$, by Corollary \ref{cor:isoextz(h)},  we have to classify the $2$-data $(H,\theta, a,\tau)$ with $\theta\not={\rm id}$ up to equivalence. 
 We can assume that $H^2(H,k^\times)\simeq \mathbb Z_2$, otherwise the result follows from Theorem \ref{cor:isoext2}. 
 Fix a set $\{\theta_1, \ldots , \theta_r\}$ of representative of the elements of ${\rm CC}_2^\bullet(\Aut(H))\simeq {\rm CC}_2^\bullet(H)$, and for each $i$, fix a non-trivial $2$-cocycle $\tau_i\in H^{2}(H,k^\times)$ such that $\tau_i\cdot \tau_i\circ \theta_i\times\theta_i=1$
 (these cocycles exist since $H^2(H,k^\times)\simeq \mathbb Z_2$). Then Corollary \ref{cor:basiciso} ensures that any $2$-data with non-trivial underlying isomorphism is equivalent to one in the list
 $$\{(H,\theta_i, a, 1), \ i=1,\dots, r, \ a\in \widehat{H}\}, \quad \{(H,\theta_i, a,\tau_i), \ \ i=1,\dots, r, \ a\in \widehat{H}\}.$$ 
 Any two different data inside one of the two sets are not equivalent by Lemma \ref{lem:Sn}, while two data taken from the two different sets are easily seen not to be equivalent either. This concludes the proof.
\end{proof}

\subsection{Back to graded twisting}
To finish the section, we go back to graded twistings. 

\begin{proposition}\label{prop:gradextexp}
	Let $(i,\alpha)$ be a cocentral action of $\mathbb{Z}_m$ on a finite group $G$. Put $H=G/i(\widehat{\mathbb{Z}_m})$,  fix a $2$-cocycle $\tau_0 : H\times H \to \widehat{\mathbb{Z}_m}$ such that
$G\simeq H\times_{\tau_0} \widehat{\mathbb{Z}_m}$ and a generator $g$ of $\mathbb Z_m$. Define a $2$-cocycle $\tau : H\times H \to \mu_m$ by $\tau(x,y)=\tau_0(x,y)(g)$, and let $\theta$ be the automorphism of $H$ induced by $\alpha=\alpha_g$
Then there exists $a : H \to \mu_m$ such that $(H,\theta, a,\tau)$ is an $m$-datum and $\Oo(G)^{i,\alpha}\simeq A_m(H,\theta, a,\tau)$.
\end{proposition}

\begin{proof}
We can assume without loss of generality that $G=H\times_{\tau_0}\widehat{\mathbb Z_m}$ and that $i$ is the canonical injection. Indeed, consider the isomorphism
$F : G \to H\times_{\tau_0} \widehat{\mathbb{Z}_m}$ making the following diagram commutative
$$\xymatrix{
1 \ar[r] & \widehat{\mathbb Z_m} \ar[r]^i \ar@{=}[d]& G \ar[r]^{\pi} \ar[d]^F &  H  \ar[r] \ar@{=}[d]& 1\\
1 \ar[r] &  \widehat{\mathbb Z_m} \ar[r]^-{i_0} &   H\times_{\tau_0} \widehat{\mathbb{Z}_m} \ar[r]^-{\pi_0} & H \ar[r] & 1
}$$
where $\pi$ is the canonical surjection, and $i_0$ and $\pi_0$ denote the canonical injection and surjection. Using the Hopf algebra isomorphism $\Oo(G) \simeq \Oo( H\times_{\tau_0} \widehat{\mathbb{Z}_m})$ induced by $F$, we obtain an isomorphism $\Oo(G)^{i,\alpha}\simeq \Oo( H\times_{\tau_0} \widehat{\mathbb{Z}_m})^{i_0,F\alpha F^{-1}}$. 

 Recall from Subsection \ref{subsec:groupprelim} (particularly the proof of Lemma \ref{lemm:autocentral}) that $\alpha=\alpha_g$ has the form $\alpha=(\theta,\mu)$ with $\theta \in \Aut(H)$ and $\mu : H \to \widehat{\mathbb Z_m}$ satisfying
$$\theta^m={\rm id}, \ \prod_{i=0}^{m-1}\mu\circ \theta^i=1, \ \tau_0 =\partial(\mu)\cdot (\tau_0\circ \theta \times \theta).$$
Define now a map $a_0 : H \to \widehat{\mathbb{Z}_m}$:
	$$a_0 =\prod_{k=1}^{m-1}(\mu\circ\theta^{-k})^{k}.$$
We then have 	
	$$\prod_{i=0}^{m-1}\tau_0\circ \theta^i\times \theta^i= \prod_{i=0}^{m-1}\tau_0\circ \theta^{-i}\times \theta^{-i}=\partial (a_0^{-1}) \ {\rm and} \ a_0\circ \theta =a_0.$$ 
Defining then $a : H \to \mu_m$ by $a(x)=a_0(x)(g)$, we get an $m$-datum $(H,\theta,a,\tau)$ satisfying the announced conditions, and we have to show that 	$A_m(H,\theta,a,\tau)\simeq \Oo(H\times_{\tau_0} \widehat{\mathbb Z_m})^{i,\alpha}$. 

For this, first note that the $\mathbb Z_m$-grading on $\Oo(H\times_{\tau_0} \widehat{\mathbb Z_m})^{i,\alpha}$ is given by 
$$\Oo(H\times_{\tau_0} \widehat{\mathbb Z_m})^{i,\alpha}_{h}= \{\phi \in \Oo(H\times_{\tau_0} \widehat{\mathbb Z_m}) \ | \ \phi(x,\chi)=\chi(h)\phi(x,1),\ \forall (x,\chi)\in H\times \widehat{\mathbb Z}_m \}.$$
Put, for $x \in H$,
$$u_g =\sum_{x \in H} \sum_{\chi \in \widehat{\mathbb{Z}_m}}\chi(g)e_{x,\chi}\in \Oo(H\times_{\tau_0} \widehat{\mathbb Z_m})^{i,\alpha}_{g}, ~\ ~\  \ e'_{x}= \sum_{\chi \in \widehat{\mathbb{Z}_m}}e_{x,\chi}\in \Oo(H\times_{\tau_0} \widehat{\mathbb Z_m})^{i,\alpha}_{e}.$$
Using the product in $\Oo(H\times_{\tau_0} \widehat{\mathbb Z_m})^{i,\alpha}$, we see that
$$u_ge'_x=e'_{\theta(x)}u_g, ~\ ~\ \ u_g^m= a.$$
Hence there exists an algebra map $A_m(H,\theta, a , \tau)\to \Oo(H\times_{\tau_0} \widehat{\mathbb Z_m})^{i,\alpha}$ sending $e_x$ to $e'_x$ and $g$ to $u_g$, which is, exactly as in the proof of Proposition \ref{prop:ext->a(h)}, a Hopf algebra isomorphism.
\end{proof}

\begin{rem}\label{rem:gradedtwisttype}
Say that an $m$-datum $(H,\theta, a, \tau)$ is of graded twist type if $\tau$ has values into $\mu_m$ and if there exists $\mu : H \to \mu_m$ such that
\[
 \prod_{i=0}^{m-1}\mu\circ \theta^i=1,~\ ~\ \ \tau =\partial(\mu)\cdot (\tau\circ \theta \times \theta), ~\ ~\ \ a =\prod_{k=1}^{m-1}(\mu\circ\theta^{-k})^{k}.
\]
The previous result (and its proof) says  that if $(i,\alpha)$ is a cocentral action of $\mathbb{Z}_m$ on a finite group $G$, then letting $H=G/i(\widehat{\mathbb{Z}_m})$, we have $\Oo(G)^{i,\alpha}\simeq A_m(H,\theta, a,\tau)$ for some $m$-datum $(H,\theta, a,\tau)$ of graded twist type.

Conversely, it is not difficult to show that if $(H,\theta, a,\tau)$ is an $m$-datum of graded twist type, then $A_m(H,\theta, a,\tau)$ is a graded twist of $\Oo(H\times_{\tau} \mu_m)$. 
\end{rem}

We now use the previous considerations to get another isomorphism result for graded twists of function algebras on finite groups by $\mathbb Z_p$, where $p$ is a prime number. We start with a lemma.

\begin{lemma}
	Let $(H,\theta,a ,\tau)$ be a $p$-datum, with $p$ a prime number. Assume that $H^2(H,k^\times)\simeq \mathbb Z_p$. Then we have $[\tau]=[\tau \circ \theta \times \theta]$ in $H^2(H,k^\times)$. 
\end{lemma}

\begin{proof}
We can assume that $\tau$ is nontrivial, hence that $[\tau]$ is a generator of $H^2(H,k^\times)$.	The group ${\rm Aut}(H)$ acts on the cyclic group  $H^2(H,k^\times )$ by automorphisms, hence there exits $l$ prime to $p$ such that $[\tau]^l=[\tau \circ \theta \times \theta]$. The assumption that we have a $p$-datum now gives 
	\[ [1]=\prod_{k=0}^{p-1} [\tau \circ \theta^k\times \theta^k]=\prod_{k=0}^{p-1} [\tau]^{l^k}=
	[\tau]^{\sum_{k=0}^{p-1}l^k}.\]
Since $p$ is prime and $[\tau]$ has order $p$, we get $l\equiv 1 [p]$, and hence  $[\tau]=[\tau \circ \theta \times \theta]$ in $H^2(H,k^\times)$. 
	\end{proof}

We arrive at our expected isomorphism result.

\begin{theorem}\label{thm:secondmain}
	Let $G$ be a finite group with cyclic center, let $(i,\alpha)$ and $(j,\beta)$ be cocentral actions of $\mathbb Z_p$  on $G$, where $p$ is a prime number, and put $H=G/i(\widehat{\mathbb Z_p})=G/j(\widehat{\mathbb Z_p})$. Assume that $\Hom(H,\mathbb Z_p)=\{1\}$ and that $H^2(H,k^\times)$ is trivial or cyclic of order $p$. Then the following assertions are equivalent.
	\begin{enumerate}
		\item The Hopf algebras $\Oo(G)^{i,\alpha}$ and $\Oo(G)^{j,\beta}$ are isomorphic.
		\item The cocentral actions $(i,\alpha)$ and $(j,\beta)$ are equivalent.
	\end{enumerate}
\end{theorem}

\begin{proof} Just as in the proof of Theorem \ref{thm:firstmain},  we have $i(\widehat{\mathbb Z_p})=j(\widehat{\mathbb Z_p})$, and $(2)\Rightarrow (1)$ follows from Lemma \ref{lemm:equi-iso}. It remains to show that $(1)\Rightarrow (2)$. 

Assume that (1) holds. To prove (2), we can safely assume that $G= H\times_{\tau_0} \widehat{\mathbb{Z}_p}$ for a $2$-cocycle $\tau_0 : H\times H \to \widehat{\mathbb{Z}_p}$ and that $i$ and $j$ are the canonical injections.
Indeed, recall  from the beginning of the proof of Proposition \ref{prop:gradextexp}, of which we retain the notation, that fixing an appropriate isomorphism $F : G \to H\times_{\tau_0} \widehat{\mathbb Z_p}$,  we get isomorphisms
$$\Oo(G)^{i,\alpha}\simeq \Oo( H\times_{\tau_0} \widehat{\mathbb{Z}_m})^{i_0,F\alpha F^{-1}}, \ ~\ ~\ 
\Oo(G)^{j,\beta}\simeq \Oo( H\times_{\tau_0} \widehat{\mathbb{Z}_m})^{i_0,F\beta F^{-1}}$$
where $i_0$ is the canonical injection. The cocentral actions $(i,\alpha)$ and $(j,\beta)$ then are equivalent if and only if the cocentral actions $(i_0, F\alpha F^{-1})$ and $(i_0,F\beta F^{-1})$ are. 

By Proposition \ref{prop:gradextexp}, we have $\Oo(G)^{i,\alpha}\simeq A_p(H,\theta, a,\tau)$
and $\Oo(G)^{j,\beta}\simeq A_p(H,\theta', a',\tau)$, for $\theta = \overline{\alpha_g}$, $\theta' =\overline{\beta_g}$
(denoting again by $f \mapsto \overline{f}$ the group morphism $\Aut_{i(\widehat{\Gamma})}(G) \to \Aut(H)$ of Lemma \ref{lemm:autocentral})
and $a,a' : H \to\mu_p$ such that 
$(H,\theta,a,\tau)$ and $(H,\theta', a',\tau)$ are $p$-data.

Since $A_p(H,\theta, a,\tau) \simeq A_p(H,\theta',a',\tau)$, Theorem \ref{thm:isoext}, which is applicable by Lemma \ref{lemm:invariantuniv},
provides a group automorphism $f\in {\rm Aut}(H)$, $\varphi : H \to k^\times$ and $l$ prime to $p$ such that 
$$\theta'^l=f\circ \theta \circ f^{-1}, \ ~\ ~\  \prod_{k=0}^{l-1}\tau\circ\theta'^{-k}\times \theta'^{-k}= \tau\circ (f^{-1}\times f^{-1}) \cdot \partial( \varphi).$$
The previous lemma ensures that $[\tau\circ\theta'\times \theta']=[\tau]$, hence we have 
$[\tau]^l=[\tau \circ f^{-1}\times f^{-1}]$ in $H^2(H,k^\times)$. Our assumptions ensure, by the universal coefficient theorem, that $H^2(H,\mathbb Z_p)\simeq \mathbb Z_p$ and that the natural map $H^2(H,\mu_p)\to H^2(H,k^\times)$ is an isomorphism, because of the exact sequence induced by the $p$-power map $k^\times \to k^\times$ 
$$1 \to {\rm Hom}(H,\mu_p) \to {\rm Hom}(H,k^\times) \to {\rm Hom}(H,k^{\times}) \to H^2(H,\mu_p) \to H^2(H, k^\times) \to H^2(H,k^\times)$$
Thus we have $[\tau]^l=[\tau \circ f^{-1}\times f^{-1}]$ in $H^2(H,\mu_p)$, and $[\tau_0]^l=[\tau_0 \circ f^{-1}\times f^{-1}]$ in $H^2(H,\widehat{\mathbb Z_p})$. Hence by Lemma \ref{lemm:autocentral2} there exists $F\in \Aut(G)$ such that
$\beta_g^l=F^{-1}\alpha_gF$ and $F_{|\widehat{\mathbb Z_p}}=(-)^l$, therefore means that our cocentral actions are equivalent.
	\end{proof}

\begin{rem}\label{rem:ccweak}
 Let $(i,\alpha)$ be a cocentral action of $\mathbb Z_m$ (of which we fix a generator $g$) on a finite group $G$. Then the Hopf algebra $\Oo(G)^{i,\alpha}$ is noncommutative if and only if $\theta$, the automorphism of $H=G/i(\mathbb Z_m)$ induced by $\alpha_g$, is non-trivial. This follows from the combination of Proposition \ref{prop:gradextexp} and of Proposition \ref{prop:a(h)co} (but can be proved quite directly as well by analyzing the $1$-dimensional representations of $\Oo(G)^{p,\alpha}$). Hence, in the situation of Theorem \ref{thm:firstmain} (or of Theorem \ref{thm:firstmainbis} for $m=2$), there is a bijection between
\begin{enumerate}
 \item the set of isomorphism classes of Hopf algebras that are noncommutative graded twisting of $\Oo(G)$ by $\mathbb Z_m$,
\item the set of equivalence classes of cocentral actions of $\mathbb Z_m$ on $G$ that  do not induce the identity on $H$, with $H$ the quotient of $G$ by its unique central subgroup of order $m$, 
\item the set of weak equivalence classes of cocentral actions of $\mathbb Z_m$ on $G$ that are not weakly equivalent to the trivial one.
\end{enumerate}
The second set is in bijection with $\mathbb X_m^\bullet(G)$ (see the end of subsection 2.3) and for $m=2$, is as well in bijection with ${\rm CC}_2^\bullet({\rm Aut}(H))$ (see Lemma \ref{lemm:autocentral}).

Under the assumptions of Theorem \ref{thm:secondmain}, we obtain, for $p$ prime, a bijection between 
\begin{enumerate}
 \item the set of isomorphism classes of Hopf algebras that are noncommutative graded twisting of $\Oo(G)$ by $\mathbb Z_p$,
\item the set of equivalence classes of cocentral actions of $\mathbb Z_p$ on $G$ that are not  equivalent to the trivial one.
\end{enumerate}
The latter set is, by Lemma \ref{lem:XmG},  in bijection with $\mathbb X_p^\bullet(G)$ (see the end of Subsection 2.3).
\end{rem}

 \section{Examples}

In this section we apply the previous results to  examine the examples announced in the introduction.

\subsection{Special linear groups over finite fields} We begin by examining graded twistings of linear groups over finite fields.

\begin{theorem}\label{thm:slz2}
Let $q=p^\alpha$, with $p\geq 3$ a prime number and $\alpha \geq 1$,
and let $n\geq 2$ be even. 
 There is a bijection between the set of isomorphism classes of noncommutative Hopf algebras that are graded twistings of $\Oo({\rm SL}_n(\mathbb{F}_q))$ by $\mathbb Z_2$ and the set $\mathbb X_2^\bullet({\rm SL}_n(\mathbb{F}_q))$.
\end{theorem}

\begin{proof}
The center of ${\rm SL}_n(\mathbb{F}_q)$ is cyclic and has even order, the character group of ${\rm SL}_n(\mathbb{F}_p)/\{\pm 1\}$ is trivial, and $H^2({\rm PSL}_n(\mathbb F_q), k^\times)$ is always cyclic under our assumptions (see \cite[Chapter 7]{kar}, for example), hence Theorem \ref{thm:firstmainbis} and Remark \ref{rem:ccweak} provide the announced bijection. 
\end{proof}

\begin{theorem}\label{thm:slzp}
Let $q=p^\alpha$, with $p$ a prime number and $\alpha \geq 1$,
let $n\geq 2$ and assume that  $m= {\rm GCD}(n,q-1)$ is prime and that $(n,q)\not\in\{(2,9), (3,4)\}$. 
Then there is a bijection between the set of isomorphism classes of noncommutative Hopf algebras that are graded twistings of $\Oo({\rm SL}_n(\mathbb{F}_q))$ by $\mathbb Z_m$ and the set $\mathbb X_m^\bullet({\rm SL}_n(\mathbb{F}_q))$.
\end{theorem}

\begin{proof}
The center of ${\rm SL}_n(\mathbb{F}_q)$ is $\mu_n(\mathbb F_q)$ and is cyclic of order $m={\rm GCD}(n,q-1)$,  the group  ${\rm Hom}({\rm PSL}_n(\mathbb{F}_p),\mathbb Z_m)$ is trivial, and $H^2({\rm PSL}_n(\mathbb F_q), k^\times)\simeq \mathbb Z_m$  under our assumptions (see \cite[Chapter 7]{kar}, for example). Hence Theorem \ref{thm:secondmain} and Remark \ref{rem:ccweak} provide the announced bijection. 
\end{proof}

In the case $n=2$, we have results for abelian cocentral extensions as well.

\begin{theorem}\label{sl2p}
Let $p\geq 3$ be a prime number.
\begin{enumerate}
 \item There are exactly $2$ isomorphism classes of noncommutative Hopf algebras that are graded twistings of $\Oo({\rm SL}_2(\mathbb{F}_p))$.
\item If $p\geq 5$, there are exactly $4$  isomorphism classes of noncommutative Hopf algebras fitting into an abelian cocentral extension $k \to \Oo({\rm PSL}_2(\mathbb{F}_p))\to A \to k\mathbb Z_2\to  k$.
\end{enumerate}
\end{theorem}

\begin{proof}
Theorem \ref{thm:slz2} ensures that there is a bijection  between the set isomorphism classes of noncommutative Hopf algebras that are graded twistings of $\Oo({\rm SL}_2(\mathbb{F}_p))$ and $\mathbb X_2^\bullet({\rm SL}_2(\mathbb F_p))$.
All the automorphisms of ${\rm SL}_2(\mathbb F_p)$ are obtained by conjugation of a matrix in ${\rm GL}_2(\mathbb F_p)$ (see e.g. \cite{dieu}), and we see that there are two equivalence classes of elements in $\mathbb X_2^\bullet({\rm SL}_2(\mathbb F_p)))$, represented by the automorphisms
$${\rm ad}\left(\begin{pmatrix}
            1 & 0 \\
0 & -1
           \end{pmatrix}\right), ~\ ~\ \ {\rm ad}\left(\begin{pmatrix}
            0 & \lambda \\
1 & 0
           \end{pmatrix}\right)$$
where $\lambda$ is a chosen element such that $\lambda\not \in (\mathbb F_p^*)^{2}$. This proves the first assertion.

We have, for $p\geq 5$, $\widehat{{\rm PSL}_2(\mathbb{F}_p)}=\{1\}$, and since $Z({\rm PSL}_2(\mathbb{F}_p))=\{1\}$ and $H^2({\rm PSL}_2(\mathbb{F}_p),k^\times) \simeq\mathbb Z_2$, the second assertion follows from the previous discussion and Corollary \ref{cor:isoext2}.
\end{proof}


\subsection{Alternating and symmetric groups} 
We now discuss examples involving alternating and symmetric groups. We begin with alternating groups and their Schur covers (see e.g. \cite{kar}).


\begin{theorem}\label{thm:an}
Let $n\geq 4$ and let  $\widetilde{A_n}$ be the unique Schur cover of the alternating group $A_n$.
\begin{enumerate}
 \item 
 There is a bijection between the set of isomorphism classes of noncommutative Hopf algebras that are graded twistings of $\Oo(\widetilde{A_n})$ by $\mathbb Z_2$ and  ${\rm CC}_2^\bullet({\rm Aut}(A_n))$. For $n\not=6$, there are precisely $\lfloor \frac{n}{2} \rfloor$ such isomorphism classes.
\item For $n=5$ or $n\geq 8$,  there is a bijection between the set of isomorphism classes of noncommutative Hopf algebras fitting into an abelian cocentral extension $k \to \Oo(A_n)\to A \to k\mathbb Z_2\to  k$ and the set ${\rm CC}_2^\bullet({\rm Aut}(A_n))\times \mathbb Z_2$. There are precisely $2\lfloor \frac{n}{2} \rfloor$ such isomorphism classes.
\end{enumerate}
\end{theorem}

\begin{proof}
In all cases  $\mathbb Z_2 \subset  Z(\widetilde{A_n})$, the center $Z(\widetilde{A_n})$ is cyclic, and $H^2(A_n,k^\times)$ is cyclic (isomorphic to $\mathbb Z_6$ for $n=6,7$ and to $\mathbb Z_2$ otherwise) and we have ${\rm Hom}(A_n,\mathbb Z_2)=\{1\}$, so the first statement is a direct consequence of Theorem \ref{thm:firstmainbis}. We have  ${\rm CC}_2^\bullet({\rm Aut}(A_n))={\rm CC}_2^\bullet({\rm Aut}(S_n))$, and when $n\not=6$ this coincides with  ${\rm CC}_2^\bullet(S_n)$, which has $\lfloor \frac{n}{2} \rfloor$ elements.



For $n=5$ or $n\geq 8$, we have moreover $H^2(A_n,k^\times)\simeq \mathbb Z_2$, and $\widehat{A_n}=\{1\}$, and since $Z(A_n)=\{1\}$,  the statement follows from Corollary \ref{cor:isoext2}.
\end{proof}

\begin{theorem}\label{thm:sn}
Assume that  $n\not=6$.
\begin{enumerate}
 \item There are exactly $4\lfloor \frac{n}{2} \rfloor$  isomorphism classes of noncommutative Hopf algebras fitting into an abelian cocentral extension $k \to \Oo(S_n)\to A \to k\mathbb Z_2\to  k$.
 \item Let $G$ be any group fitting into a central extension $1 \to \mathbb Z_2 \to G \to S_n \to 1$. There are exactly  $2\lfloor \frac{n}{2} \rfloor$  isomorphism classes of noncommutative Hopf algebras that are graded twistings of $\Oo(G)$ by $\mathbb Z_2$.
\end{enumerate}
\end{theorem}

\begin{proof} Every automorphism of $S_n$ is inner when $n\not=6$, and we have $\widehat{S_n} \simeq \mathbb Z_2\simeq H^2(S_n,k^\times)$, so the first assertion follows from Theorem \ref{cor:isoext4}.
	
	Let $G$ be a group as in the statement. By Proposition \ref{prop:gradextexp}, a graded twisting of $\Oo(G)$ is isomorphic to 
	$A_2(S_n,\theta,a, \tau)$ for a cocycle $\tau : S_n\times S_n \to \mathbb Z_2$ canonically build from the central extension $1 \to \mathbb Z_2 \to G \to S_n \to 1$. Hence Lemma \ref{lem:Sn} ensures that there are at most $2\lfloor \frac{n}{2} \rfloor$  isomorphism classes of noncommutative graded twistings of $\Oo(G)$.

Conversely, start with a $2$-datum $(S_n,\theta,a, \tau)$, with $\tau$ as before.  We wish to prove that $A_2(S_n,\theta,a,\tau)$ is isomorphic to a graded twist of $\Oo(G)$.
By Lemma \ref{lem:tauinner}, since any automorphism of $S_n$ is inner, there exists $\mu : S_n\to \mu_2$ such that $\tau\cdot \tau\circ\theta\times \theta=\partial(\mu)$. Then $a^{-1}$ and $\mu$ differ by an element of $\widehat{S_n}$, and hence $a^2=1$. Our $2$-data $(S_n,\theta,a,\tau)$ is then of graded twist type as in Remark  \ref{rem:gradedtwisttype}, and then we know that $A_2(S_n,\theta,a,\tau)$ is a graded twist of $\Oo(H\times_\tau\mu_2)\simeq \mathcal O(G)$. This concludes the proof.
	\end{proof}

\subsection{The alternating group $A_5$} Examples with the alternating group $A_5$ fall into the series studied in the last two subsections, but there is a special interest in $A_5$, because of the following result from \cite{bina}: any finite-dimensional cosemisimple Hopf algebra $A$ having a faithful $2$-dimensional comodule $V$ with $V\otimes V^*\simeq V^*\otimes V$ fits into an abelian cocentral extension 
$$k \to \mathcal O(H) \to A \to k\mathbb Z_m\to k$$
for some $m\geq 2$ and some polyhedral group $H\in \{ A_4, \ S_4, \ A_5, \ D_{2n}\}$ . Using Corollary \ref{cor:isoext2} and the easy description of the conjugacy classes in $S_5\simeq \Aut(A_5)$, we have the following contribution to this situation.

\begin{theorem}
 Let $m\geq 2$ and let $N$ be the number of isomorphism classes of noncommutative Hopf algebras $A$ fitting into
an abelian cocentral extension 
$k \to \mathcal O(A_5) \to A \to k\mathbb Z_m\to k$.  
Then, according to the value of ${\rm GCD}(m,120)$, the value of $N$ is as follows:
\begin{enumerate}
 \item $N=0$ if ${\rm GCD}(m,120)=1$.
\item $N=4$ if ${\rm GCD}(m,120)=2$.
\item $N=1$ if ${\rm GCD}(m,120)=3, 5$.
\item $N=6$ if ${\rm GCD}(m,120)=4,8$.
\item $N=7$ if ${\rm GCD}(m,120)=6, 20, 40$.
\item $N=5$ if ${\rm GCD}(m,120)=10$.
\item $N=9$ if ${\rm GCD}(m,120)=12, 24$.
\item $N=2$ if ${\rm GCD}(m,120)=15$.
\item $N=10$ if ${\rm GCD}(m,120)=30, 60, 120$.
\end{enumerate}
\end{theorem}

Of course, the above theorem does not give any information about the realizability of one of the above Hopf algebras
as Hopf algebras having a faithful $2$-dimensional comodule.

\subsection{Dihedral groups $D_n$}
In this subsection we discuss Hopf algebras fitting into an abelian cocentral extension
		$$k \to \Oo(D_n) \to A \to k\mathbb Z_2 \to k$$
with $D_n$ the dihedral group of order $2n$. While the group structure of $D_n$ is certainly less rich than the one of the groups of the previous sections, the situation with Hopf algebra extensions as above is in fact much more involved.

\subsubsection{Notation} As usual, the group $D_n$ is presented by generators $r$, $s$ and relations $r^n=1=s^2$, $sr=r^{n-1}s$, and its automorphisms all are of the form $\Psi_{k,l}$, $(k,l) \in \mathbb Z/n\mathbb Z \times U(\mathbb Z/n\mathbb Z)$, with 
$$\Psi_{k,l}(r) = r^l, \quad \Psi_{k,l}(s) = sr^k.$$
Such an automorphism $\Psi_{k,l}$ has order $2$ precisely when $(k,l)\not=(0,1)$, $l^2=1$ and $k(l+1)=0$ (in $\mathbb Z/n\mathbb Z$). The following facts are also well-known:
\[
\text{if $n$ is odd, then }  Z(D_n)=\{1\}, ~\ ~\ \ H^2(D_n,k^\times)=\{1\}, ~\ ~\ \ \widehat{D_n}\simeq \mathbb Z_2,
\]
\[
\text{if $n$ is even, then }  Z(D_n)=\{1, ~\ ~\ r^{n/2}\}, ~\ ~\ \ H^2(D_n,k^\times)\simeq \mathbb Z_2, ~\ ~\ \ \widehat{D_n}\simeq \mathbb Z_2\times \mathbb Z_2
\] 

\subsubsection{The case when $n$ is odd} Here the situation is very simple, since we are in the situation of Corollary \ref{cor:isoext3}: we have, for $m\geq 1$, a bijection between 
the set of isomorphism classes of noncommutative Hopf algebras $A$ fitting into an abelian cocentral extension
		 $$k \to \Oo(D_n) \to A \to k\mathbb Z_m \to k$$
and   \begin{enumerate}
	\item if $m$ is odd, the set ${\rm CC}_m^\bullet({\rm Aut}(D_n))$;
	\item if $m$ is even, the set ${\rm CC}_m^\bullet({\rm Aut}(D_n))\times \widehat{D_n}$.
\end{enumerate}



An immediate application yields the following result.

\begin{theorem}
 Let $n\geq 3$ be odd and let $e_n$ be the number of isomorphism classes of noncommutative Hopf algebras $A$ fitting into an abelian cocentral extension $k \to \Oo(D_{n}) \to A \to k\mathbb Z_2 \to k$.
\begin{enumerate}
 \item If $n = p^r$ with $p$ odd prime and $r \geq 1$, then $e_n=2$.
\item If $n=p^rq^s$, with $p,q$ distinct odd primes and $r,s\geq 1$, then $e_n=6$.
\end{enumerate}
\end{theorem}

\begin{proof}
The previous statement ensures that $e_n$ is twice the number of conjugacy classes of elements of order $2$ in ${\rm Aut}(D_{n})$, that we compute in the above two cases.
In the first case there is precisely one such conjugacy class, represented by
 $\Psi_{0,-1}$. In the second situation, fix integers $a,b$ such that $p^ra+q^s b=1$, and such that $a,b$ become invertible in $\mathbb Z/p^rq^s\mathbb Z$. One checks that there are $3$ conjugacy classes of elements of order $2$ in ${\rm Aut}(D_{p^rq^s})$, represented by $\Psi_{0,-1}$, $\Psi_{0,2q^sb-1}$ and $\Psi_{0,2p^ra-1}$.
\end{proof}

\begin{rem}
For $n=3$,  the two non-isomorphic Hopf algebras of the previous theorem are the two non-isomorphic  noncommutative and noncocommutative Hopf algebras of dimension $12$, classified by Fukuda \cite{fuk}.
\end{rem}

\subsubsection{The case when $n$ is even} We now assume, throughout the subsection, that $n$ is even. None of our previous classification results apply here and we have to perform a specific analysis.
We obtain a pretty satisfactory result in Table \ref{Tab:Dn}, which, on the other hand, indicates that, in full generality, it is probably hopeless to get compact classification results, such as in theorems  \ref{cor:isoext2}, \ref{cor:isoext3}, \ref{cor:isoext4}.

 We begin with a useful test to determine whether a $2$-cocycle on $D_n$ is trivial or not, and when it is trivial, to describe it as an explicit coboundary.

\begin{lemma} \label{lem:trivcocDn}
Let $\beta \in Z^{2}(D_n,k^\times)$. The following assertions are equivalent:
\begin{enumerate}
 \item $[\beta]=1$ in $H^2(D_n,k^\times)$;	
\item   there exist $x,y\in k^\times$ such that
\[   x^n = \beta(r,r)\beta(r,r^2)\cdots \beta(r,r^{n-1}) , \quad
y^2= \beta(s,s), \quad x^{2} = \beta(r,r^{n-1})\beta(r^{n-1},s)^{-1}\beta(s,r); \]
\item  we have $\left(\beta(r,r^{n-1})\beta(r^{n-1},s)^{-1}\beta(s,r) \right)^{n/2} = \beta(r,r)\beta(r,r^2)\cdots \beta(r,r^{n-1})$.
\end{enumerate}
Moreover, when 	$[\beta]=1$ in $H^2(D_n,k^\times)$, picking $x,y\in k^\times$ as above, the map $\mu : D_n\to k^\times$ defined by, for $0\leq i\leq n-1$, $0\leq j\leq 1$,
\[ \mu(r^is^j) = \beta(r^{i},s^j)^{-1}\beta(r,r)^{-1}\beta(r,r^2)^{-1}\cdots  \beta(r,r^{i-1})^{-1}\beta(s,s^j)^{-1}x^iy^j\]
is such that $\beta=\partial(\mu)$.
\end{lemma}	 

\begin{proof}
 This is a direct verification, using the well-known fact that $\beta$ is trivial if and only if there exists an algebra map $k_\beta D_n \to k$, where $k_\beta D_n$ is the twisted group algebra. Such an algebra map then furnishes a map $\mu$ with $\beta=\partial(\mu)$. 
\end{proof}

We now exhibit a convenient explicit non-trivial $2$-cocycle over $D_n$.

\begin{lemma}\label{lem:tauomega}
 Let $\omega\in k^\times$ be such that $\omega^n=1$. Then the map 
\begin{align*}
 \tau_\omega : D_n \times D_n &\longrightarrow k^\times \\
(r^is^j, r^ks^l) &\longmapsto \omega^{jk} \ (j,l  \in \{0,1\})
\end{align*}
is a $2$-cocycle, and $[\tau_\omega]=1\iff \omega^{n/2}=1$. When $\omega^{n/2}=-1$, $\tau_\omega$ represents the only non-trivial cohomology class in $H^2(D_n,k^\times)$. 
\end{lemma}

\begin{proof}
 It is a straightforward verification that $ \tau_\omega$ is  a $2$-cocycle, and the triviality condition follows from Lemma \ref{lem:trivcocDn}. The last assertion follows from the previous one and the fact that $H^2(D_n,k^\times)\simeq \mathbb Z_2$.
\end{proof}

We now proceed to describe the possible $2$-data over $D_n$. We begin with a preliminary lemma.

\begin{lemma}\label{lem:acircthetaDn}
Let $\theta \in {\rm Aut}(D_n)$ and $\tau \in Z^{2}(D_n,k^\times)$ be such that $[\tau]=1$ and $\tau \circ\theta\times \theta=\tau$, and let $a : D_n \to k^\times$ be such that $\tau=\partial(a)$. If  $a(\theta(r))=a(r)$ and $a(\theta(s))=a(s)$, then $a \circ \theta =a$.
\end{lemma}

\begin{proof}
	We have for any $g,h\in D_n$, 
	\[ a(g)a(h)a(gh)^{-1}=\tau(g,h)= \tau(\theta(g),\theta(h))=a(\theta(g)) a(\theta(h)) a(\theta(gh))^{-1}\]
	hence if $a(g)=a(\theta(g))$ 	and  $a(h)=a(\theta(h))$, we have $a(\theta(gh))=a(gh)$, and the result follows since $D_n$ is generated by $r$ and $s$.
\end{proof}

\begin{lemma}\label{lem:possa}
 Let $\Psi_{u,v}\in {\rm Aut}(D_n)$. Let $\omega\in k^\times$ with $\omega =-1$ if $n/2$ is odd, and with $\omega$ a primitive nth root of unity if $n/2$ is even. Let $\tau_\omega\in Z^2(D_n,k^\times)$ be the non trivial cocycle of Lemma \ref{lem:tauomega}.   Let $x,y\in k^\times$ be such that $x^n=1$, $y^2=\omega^{u}$ and $x^2=\omega^{-v-1}$ ($x^2=1$ if $n/2$ is odd).  The map $a_{x,y} : D_n\to k^{\times}$ defined by
$$ a_{x,y}(r^is^j)=\omega^{-uj}x^{i}y^{j},  \  ~\  ~\  0\leq i\leq n-1, \ 0\leq j \leq 1$$
is such that $\tau_\omega(\tau_\omega\circ \Psi_{u,v}\times \Psi_{u,v}) = \partial(a_{x,y}^{-1})$, and any map satisfying this identity is of the form $a_{\pm x,\pm y}$. Moreover we have $a_{x,y}\circ \Psi_{u,v}=a_{x,y}$ if and only if $x^u=1=x^{v-1}$.

Assume furthermore that $\Psi_{u,v}$ has order $2$. Then $a_{x,y}\circ \Psi_{u,v}=a_{x,y}$ if and only if we are in one of the following situations.

\begin{enumerate}
 \item $n/2$ is odd, $u$ is even, $x=\pm 1$ and $y=\pm 1$.
\item $n/2$ is odd, $u$ is odd, $x=1$ and $y=\pm \xi$, with $\xi$ a primitive fourth root of unity.
\item $n/2$ is even, $u$ is even, $v^2=1+kn$, $u(1+v) = ln$ with $k,l$ even, and $x=\pm \omega^{\frac{-v-1}{2}}$, $y=\pm \omega_0^u$, with $\omega_0^2=\omega$.
\item $n/2$ is even, $u$ is odd, $v^2=1+kn$, $u(1+v) = ln$ with $k,l$ even, and $x=\omega^{\frac{-v-1}{2}}$, $y=\pm \omega_0^u$, with $\omega_0^2=\omega$.
\item $n/2$ is even, $u$ is odd, $v^2=1+kn$, $u(1+v) = ln$ with $k$ even and $l$ odd, and $x=-\omega^{\frac{-v-1}{2}}$, $y=\pm \omega_0^u$, with $\omega_0^2=\omega$.
\end{enumerate}
\end{lemma}

\begin{proof}
The cocycle $\tau_\omega(\tau_\omega\circ \Psi_{u,v}\times \Psi_{u,v})$ is necessarily trivial since $H^2(D_n,k^\times)$ has order $2$, and Lemma \ref{lem:trivcocDn} yields the identity $\tau_\omega(\tau_\omega\circ \Psi_{u,v}\times \Psi_{u,v}) = \partial(a_{x,y}^{-1})$. Any map $D_n\to k$ satisfying the previous identity differs from $a_{x,y}$
by the multiplication of an element in $\widehat{D_n}$, and hence is of the form $a_{\pm x,\pm y}$. The previous lemma ensures that $a_{x,y}\circ \Psi_{u,v}=a_{x,y}$ if and only if 
$a_{x,y}(\Psi_{u,v}(r))=a_{x,y}(r)$ and $a_{x,y}(\Psi_{u,v}(s))=a_{x,y}(s)$. We have
$$a_{x,y}(r)= x, \ a_{x,y}(\Psi_{u,v}(r)) = x^{v}, \ a_{x,y}(s)=\omega^{-u}y, \ a_{x,y}(\Psi_{u,v}(s))=\omega^{-u}x^{-u}y.$$ Hence we have $a_{x,y}\circ \Psi_{u,v}=a_{x,y}$ if and only if $x^{v-1}=1$ and $x^u=1$. The result is then obtained via a case by case discussion and the previous lemma.
\end{proof}

Lemma \ref{lem:possa} describes the automorphisms $\Psi_{u,v}$ that fit into a $2$-datum $(D_n, \Psi_{u,v}, a, \tau_\omega)$ with the description of the possible maps $a$. We now have to classify them up to equivalence: this is done in our next lemma.

\begin{lemma}\label{lem:classa}
Let $\Psi_{u,v}\in {\rm Aut}(D_n)$ be an element of order $2$, and retain the notation of Lemma \ref{lem:possa}.
\begin{enumerate}
\item For $n/2$ odd, $u$ even and $x,y$ as in Lemma \ref{lem:possa} ($x=\pm 1$ and $y=\pm 1$), the $2$-data $(D_n,\Psi_{u,v}, a_{1,1}, \tau_\omega)$ and $(D_n,\Psi_{u,v}, a_{1,-1}, \tau_\omega)$ are equivalent, while the $2$-data $(D_n,\Psi_{u,v}, a_{1,1}, \tau_\omega)$, $(D_n,\Psi_{u,v}, a_{-1,1}, \tau_\omega)$ and $(D_n,\Psi_{u,v}, a_{-1,-1}, \tau_\omega)$ are pairwise non-equivalent. Hence there are exactly three equivalence classes of $2$-data over $D_n$ having $\Psi_{u,v}$ as underlying automorphism.
 \item For $n/2$ odd and  $u$  odd, the $2$-data $(D_n,\Psi_{u,v}, a_{1,\xi}, \tau_\omega)$ and $(D_n,\Psi_{u,v}, a_{1,-\xi}, \tau_\omega)$ are equivalent. Hence there is only one equivalence class of $2$-data over $D_n$ having $\Psi_{u,v}$ as underlying automorphism. 
\item For $n/2$ even and  $u$  odd satisfying the conditions of cases 4 or 5 in Lemma \ref{lem:possa}, and for $x,y$ as above, the $2$-data $(D_n,\Psi_{u,v}, a_{x,y}, \tau_\omega)$ and $(D_n,\Psi_{u,v}, a_{x,-y}, \tau_\omega)$ are equivalent. Hence there is only one equivalence class of $2$-data over $D_n$ having $\Psi_{u,v}$ as underlying automorphism.
\item For $n\equiv0[8]$ and  $u$ even satisfying the conditions of case 3 in Lemma \ref{lem:possa}, and for $x,y$ as in Lemma \ref{lem:possa}, the $2$-data $(D_n,\Psi_{u,v}, a_{x,y}, \tau_\omega)$ and $(D_n,\Psi_{u,v}, a_{x,-y}, \tau_\omega)$ are equivalent, while the $2$-data $(D_n,\Psi_{u,v}, a_{x,y}, \tau_\omega)$ and $(D_n,\Psi_{u,v}, a_{-x,y}, \tau_\omega)$ are not equivalent. Hence there are exactly two equivalence classes of $2$-data over $D_n$ having $\Psi_{u,v}$ as underlying automorphism. 
\item For $n\equiv4[8]$,  $u$ even and $v\equiv 3 [4]$ satisfying the conditions of case 3 in Lemma \ref{lem:possa}, and for $x,y$ in Lemma \ref{lem:possa}, the $2$-data $(D_n,\Psi_{u,v}, a_{x,y}, \tau_\omega)$ and $(D_n,\Psi_{u,v}, a_{x,-y}, \tau_\omega)$ are equivalent, while the $2$-data $(D_n,\Psi_{u,v}, a_{x,y}, \tau_\omega)$ and $(D_n,\Psi_{u,v}, a_{-x,y}, \tau_\omega)$ are not equivalent. Hence there are exactly two equivalence classes of $2$-data over $D_n$ having $\Psi_{u,v}$ as underlying automorphism. 
\item For $n\equiv4[8]$,  $u$ even and $v\equiv 1 [4]$ satisfying the conditions of case 3 in Lemma \ref{lem:possa}, and for $x,y$ as in Lemma \ref{lem:possa},  there are exactly three equivalence classes of $2$-data over $D_n$ having $\Psi_{u,v}$ as underlying automorphism. 
\end{enumerate}
\end{lemma}

\begin{proof}
Recall that an equivalence between two $2$-data $(H,\theta, a, \tau)$ and $(H,\theta, a', \tau)$ is provided by a pair $(f,\varphi)$ with $f\in \Aut(H)$ and $\varphi : H \to k^\times$ satisfying
$${\rm (a)} \quad f\circ \theta =\theta \circ f, \ {\rm (b)} \quad \varphi \cdot \varphi\circ \theta \cdot a'=a\circ f^{-1}, \quad {\rm (c)} \ \tau =\partial(\varphi)\cdot \tau\circ f^{-1}\times f^{-1}.$$
For $u$ odd (cases (2) and (3) in the lemma), taking $\varphi\in \widehat{H}$ such that $\varphi(r)=-1$, we see that the pair $({\rm id}, \varphi)$ realizes an equivalence between the $2$-data $(D_n,\Psi_{u,v}, a_{x,y}, \tau_\omega)$ and $(D_n,\Psi_{u,v}, a_{x,-y}, \tau_\omega)$, and thus the statements (2) and (3) are proved.

We assume now that $u$ is even. Let $f\in \Aut(D_n)$, with $f^{-1}=\Psi_{\alpha,\beta}$. Then, similarly to Lemma \ref{lem:possa}, one shows that the maps $\varphi : D_n \to k^\times $ satisfying (c) above are defined by
$$\varphi_{z,t}(r^is^j)= \omega^{-j\alpha}z^it^j$$
where $z=\pm \omega^{\frac{1-\beta}{2}}$, $t=\pm(\omega_0)^\alpha$, with $\omega_0^2=\omega$. We then have
$$\varphi_{z,t}(r^is^j)\varphi_{z,t}(\Psi_{u,v}(r^is^j)) = \omega^{-2j\alpha}z^{i(1+v)-ju} t^{2j}$$
and in particular
$$\varphi_{z,t}(r)\varphi_{z,t}(\Psi_{u,v}(r))=z^{1+v}= \omega^{\frac{(1-\beta)(1+v)}{2}}, \quad  \varphi_{z,t}(s)\varphi_{z,t}(\Psi_{u,v}(s))=\omega^{-\alpha}z^{-u}.$$
Equation (b), for $a_{x,y}$ and $a_{x',y'}=\varepsilon a_{x,y}$, where $\varepsilon  \in \widehat{D_n}$  (with $x'=\varepsilon (r)x$, $y'=\varepsilon(s)y)$), then becomes $$z^{1+v}=\varepsilon(r)x^{\beta-1}, \quad z^{-u}=\varepsilon(s)\omega^{\alpha}x^{-\alpha}.$$
The first equation is then 
$$\omega^{\frac{(1-\beta)(v+1)}{2}} = \varepsilon(r)\omega^{\frac{(-v-1)(\beta-1)}{2}}$$ 
which gives $\varepsilon(r)=1$. Hence if the $2$-data $(D_n, \Psi_{u,v}, a_{x,y}, \tau_\omega)$ and $(D_n, \Psi_{u,v}, a_{x',y'}, \tau_\omega)$ are equivalent, then necessarily $x=x'$, as claimed. 

Since $u$ is even, the second equation now is 
\begin{align}\label{eq2}\omega^{\frac{(\beta-1)u}{2}-\alpha}=\varepsilon(s) x^{-\alpha}.\end{align}
Assume that $n/2$ is odd, so that $\omega=-1$. Since $\beta$ is odd, the second equation becomes $(-1)^\alpha = \varepsilon(s) x^\alpha$, with $x=\pm 1$. This is possible with $\varepsilon(s)=-1$ only if $x=1$. Hence we see that the $2$-data $(D_n, \Psi_{u,v}, a_{-1,1}, \tau_\omega)$ and $(D_n, \Psi_{u,v}, a_{-1,-1}, \tau_\omega)$ are not equivalent. 
Conversely, taking $f=f^{-1}=\Psi_{n/2,1}$ (which commutes with $\Psi_{u,v}$) and $\varphi_{z,t}$ as above, we see that the pair $(\Psi_{n/2,1},\varphi_{z,t})$ makes the $2$-data $(D_n, \Psi_{u,v}, a_{1,1}, \tau_\omega)$ and $(D_n, \Psi_{u,v}, a_{1,-1}, \tau_\omega)$  equivalent. This concludes the proof of Assertion (1).

Assume now that $n/2$ is even. Then, writing $x=\nu \omega^{\frac{-v-1}{2}}$ with $\nu=\pm 1$, Equation \ref{eq2}
becomes 
\begin{align}\label{eq2.1}\varepsilon(s) = \nu^\alpha \omega^{\frac{(\beta-1)u-\alpha(v+3)}{2}} = \nu^\alpha \omega^{\frac{(\beta-1)u-\alpha(v-1)}{2}}\omega^{-2\alpha}.\end{align}
If $v\equiv 3[4]$, taking $\alpha=n/2$ and $\beta=1$,  Equation \ref{eq2.1} is realized with $\varepsilon(s) =-1$. Hence taking $f=\Psi_{n/2,1}$ (which commutes with $\Psi_{u,v}$) and $\varphi_{z,t}$ as above, we obtain that the $2$-data $(D_n,\Psi_{u,v}, a_{x,y},\tau_\omega)$ and $(D_n,\Psi_{u,v}, a_{x,-y},\tau_\omega)$ are equivalent when $v\equiv 3[4]$.
This proves Assertion (5).

Assume that $v\equiv 1[4]$. If $n/4$ is even, it is not difficult to check that the condition $v^2\equiv1[2n]$ implies that $v\equiv 1[8]$. Then we see that condition \ref{eq2.1} is realized with $\varepsilon(s)=-1$ by taking $\beta =1$ and $\alpha = n/4$, and choosing $f=\Psi_{n/4,1}$ (which commutes with $\Psi_{u,v}$) we obtain that the $2$-data  $(D_n,\Psi_{u,v}, a_{x,y},\tau_\omega)$ and $(D_n,\Psi_{u,v}, a_{x,-y},\tau_\omega)$ are equivalent. This finishes the proof of Assertion (4).

We assume finally that $n/4$ is odd (still with $v\equiv 1[4]$). Recall that $x=\nu \omega^{\frac{-v-1}{2}}$ with $\nu=\pm 1$. Taking $\alpha=n/4$ and $\beta=1$, Equation \ref{eq2.1} with $\varepsilon(s)=-1$ is realized in the following two cases:
$$\nu=1, \ v\equiv 1[8] \ ; \quad \nu=-1, \ v\equiv 5 [8].$$
Taking $f=\Psi_{n/4,1}$,  we obtain:

\begin{itemize}[label=$\circ$]
\item for $v\equiv 1[8]$,  the $2$-data  $(D_n,\Psi_{u,v}, a_{\omega^{\frac{-v-1}{2}},y},\tau_\omega)$ and $(D_n,\Psi_{u,v}, a_{\omega^{\frac{-v-1}{2}},-y},\tau_\omega)$ are equivalent,  

\item
for $v\equiv 5[8]$, the $2$-data  $(D_n,\Psi_{u,v}, a_{-\omega^{\frac{-v-1}{2}},y},\tau_\omega)$ and $(D_n,\Psi_{u,v}, a_{-\omega^{\frac{-v-1}{2}},-y},\tau_\omega)$ are equivalent. 
\end{itemize}

To see that these are the only cases where there is an equivalence, assume that Equation \ref{eq2.1} holds
with $\varepsilon(s)=-1$, and $(\beta-1)u\equiv \alpha(v-1)[n]$:
$$-1 = \nu^\alpha \omega^{\frac{(\beta-1)u-\alpha(v-1)}{2}}\omega^{-2\alpha}.$$
Squaring this identity, we see that $\alpha\in \{0,n/4, n/2, 3n/4\}$. One checks easily that the condition $u(v+1)\equiv 0[2n]$ implies that there does not exist $\beta$ such that $(\beta-1)u\equiv n [2n]$. Hence, assuming that 
$\nu=1$ and   $v\equiv 5[8]$ or that  $\nu=-1$ and  $v\equiv 1[8]$, and  examining all the possibilities for $\alpha$, we always arrive at an identity $-1=1$: contradiction. This finishes the proof of Assertion (6), hence the proof of the lemma.
\end{proof}

Lemma \ref{lem:classa} enables one to classify the reduced $2$-data over $D_n$, as soon as the representative elements for the conjugacy classes of order $2$ elements in ${\rm Aut}(D_n)$ have been found. We record the result in Table \ref{Tab:Dn}, where $\Psi_{u,v}$ is an order $2$ automorphism of $D_n$ (hence with $v^2\equiv 1[n]$ and $u(v+1)\equiv 0[n]$), and $N(u,v)$ denotes the number of equivalence classes of reduced $2$-data over $D_n$ having $\Psi_{u,v}$ as underlying automorphism.

\begin{table}[h]
\centering
\begin{tabular}{|p{5cm}|c|}
 \hline
Properties of $n/2$, $u$ and $v$ & $N(u,v)$\\
\hline
$n/2$ odd, $u$ odd & 1 \\
\hline
$n/2$ odd, $u$ even & 3 \\
\hline
$n/2$ even, $u$ odd, $v^2\equiv 1[2n]$ & 1\\
\hline
$n\equiv0[8]$, $u$ even, $v^2\equiv 1[2n]$, $u(v+1)\equiv 0 [2n]$ & 2\\
\hline 
$n\equiv4[8]$, $u$ even, $v^2\equiv 1[2n]$, $u(v+1)\equiv 0 [2n]$, $v\equiv 3[4]$ & 2\\
\hline 
$n\equiv4[8]$, $u$ even, $v^2\equiv 1[2n]$, $u(v+1)\equiv 0 [2n]$, $v\equiv 1[4]$ & 3\\
\hline
\end{tabular}
\vspace{0.5cm}
\caption{Number of reduced $2$-data over $D_n$ having $\Psi_{u,v}$ as automorphism}
\label{Tab:Dn}
\end{table}

We now apply the results in Table \ref{Tab:Dn} to enumerate the Hopf algebras fitting into a universal cocentral extension $k \to \Oo(D_{n}) \to A \to k\mathbb Z_2 \to k$ in a number of particular cases.

\begin{theorem}
 Let $n\geq 4$ be even and let $e_n$ be the number of isomorphism classes of noncommutative Hopf algebras $A$ fitting into a universal cocentral extension $k \to \Oo(D_{n}) \to A \to k\mathbb Z_2 \to k$.
\begin{enumerate}
 \item If $n = 2^r$ with $r \geq 2$, then $e_n=3$.
\item If $n=2p^r$, with $r\geq 1$ and $p$ odd prime, then $e_n=5$.
\item If $n=4p^r$, with $r\geq 1$ and $p$  odd prime, then $e_n=9$.
\item If $n=2^sp^r$, with $s\geq 3$, $r\geq 1$ and $p$  odd prime, then $e_n=10$.
\end{enumerate}
\end{theorem}

\begin{proof}
 A $2$-datum $(D_n,\theta, a, \tau)$ is not reduced if $\tau$ is a trivial cocycle, because $Z(D_n)$ is non trivial, and is reduced if $\tau$ is the non-trivial $2$-cocycle in Lemma \ref{lem:tauomega}. Hence, by Corollary \ref{cor:basiciso}, Theorem \ref{thm:isoext} and Proposition \ref{prop:a(h)co}, $e_n$ equals the number of equivalence classes of $2$-data $(D_n,\theta, a,\tau_\omega)$ with $\theta\not={\rm id}$, which now will be determined in each case using Table \ref{Tab:Dn}.

For $n=4$, there are $3$ conjugacy classes of order $2$ elements in  ${\rm Aut}(D_n)$, represented by $\Psi_{2,1}$, $\Psi_{0,-1}$ and $\Psi_{1,-1}$. The first automorphism does not satisfy the condition $u(v+1)\equiv 0[8]$ in Lemma \ref{lem:possa}, so cannot fit into a $2$-datum. For the last two automorphisms, Table \ref{Tab:Dn} gives $e_4=2+1=3$. 

For $n=2^r$ with $r\geq 3$, there are $5$ conjugacy classes of order $2$ elements in ${\rm Aut}(D_n)$, represented by:
\[
 \Psi_{2^{r-1},1}, \quad \Psi_{0,2^{r-1}-1}, \quad \Psi_{0, 2^{r-1}+1}, \quad \Psi_{0,-1}, \quad \Psi_{1,-1}.  
\]
Among these automorphism, only $\Psi_{0,-1}$ and  $\Psi_{1,-1}$ satisfy the compatibility conditions of Lemma  \ref{lem:possa} that make them part of a $2$-data. Finally, Table \ref{Tab:Dn} gives again $e_n=2+1=3$.

For $n=2p^r$, with $p$ odd prime, there are $3$ conjugacy classes of order $2$ elements in ${\rm Aut}(D_n)$, represented by  $\Psi_{p^r,1}$, $\Psi_{0,-1}$ and $\Psi_{1,-1}$. Table \ref{Tab:Dn} gives $e_n=1+3+1=5$. 

For $n=4p^r$, with $p$ odd prime, fix integers $a,b$ such that $4a+p^rb=1$, and such that $a,b$ become invertible in $\mathbb Z/4p^r\mathbb Z$. There are four elements in $\mathbb Z/4p^r\mathbb Z$ such that $v^2=1$ and in fact $v^2\equiv 1[2n]$: $v=\pm 1$, $v=\pm(4a-p^rb)$. One then checks that the representatives of the conjugacy classes of the order $2$ elements $\Psi_{u,v}\in {\rm Aut}(D_n)$ satisfying the conditions in Lemma \ref{lem:possa} are 
\[
 \Psi_{0,-1}, \quad \Psi_{1,-1}, \quad \Psi_{0,8a-1}, \quad \Psi_{p^r, 8a-1}, \quad \Psi_{0, 1-8a}.
\]
Table \ref{Tab:Dn} now yields that $e_n=2+1+2+1+3=9$.

For $n=2^sp^r$, with $p$ odd prime and $s\geq 3$, fix integers $a,b$ such that $2^sa+p^rb=1$, and such that $a,b$ become invertible in $\mathbb Z/2^sp^r\mathbb Z$. There are $8$ elements in $\mathbb Z/2^sp^r\mathbb Z$ such that $v^2=1$ but only $4$ such that $v^2\equiv 1[2n]$: $v=\pm 1$, $v=\pm(2^sa-p^rb)=\pm(2^{s+1}a-1)$. One then checks that the representatives of the conjugacy classes of the order $2$ elements $\Psi_{u,v}\in {\rm Aut}(D_n)$ satisfying the conditions in Lemma \ref{lem:possa} are 
\[
 \Psi_{0,-1}, \quad \Psi_{1,-1}, \quad \Psi_{0,2^{s+1}a-1}, \quad \Psi_{p^r, 2^{s+1}a-1}, \quad \Psi_{0, 1-2^ {s+1}a}, \quad \Psi_{2^s, 1-2^{s+1}a}.
\]
Table \ref{Tab:Dn} now yields that $e_n=2+1+2+1+2+2=10$.
\end{proof}

\begin{rem}
 Part (1) of the above theorem contributes to the classification of semisimple Hopf algebra of dimension $2^r$, studied in \cite{kas03,kas16}.
\end{rem}

\subsection{Hopf algebras of dimension $p^2q^r$} 
To conclude the paper, we  look at an example where the group $H$ is abelian, one of the most studied situation in the literature \cite{mas,nat,kro}. We wish to prove the following result, for which the case $r=1$ was obtained in \cite{nat}.


\begin{theorem}\label{thm:p2qr}
Let $p,q$ be odd prime numbers, let $r\geq 1$ and assume that $q^{r}|p-1$. 
 The number of isomorphism classes of noncommutative and noncocommutative Hopf algebras fitting into a cocentral extension $k \to \Oo(\mathbb Z_p^2) \to A \to k\mathbb Z_{q^r} \to k$ is precisely $\frac{1}{2} (\sum_{i=1}^r q^i+q^{i-1})= \frac{(q+1)(q^r-1)}{2(q-1)}$.
\end{theorem}


The rest of the section is devoted to the proof of Theorem \ref{thm:p2qr}. We begin with some generalities. Recall from Subsection \ref{subsecClass} that  if $G$ is a group and $m\geq 1$, the set ${\rm CC}_m^\bullet(G)$ is the set of  elements of $G$ such that $x^m=1$ and $x\not=1$, modulo the equivalence relation defined by $x\sim y\iff$ there exists $l$ prime to $m$ such that $x^l$ is conjugate to $y$. For $d>1$ a divisor of $m$, denote by  ${\rm CC}_{m,d}^\bullet(G)$ the set of equivalence classes of elements having order $d$ in $G$ (clearly the order of an element is well-defined in ${\rm CC}_m^\bullet(G)$). We get a decomposition
\[
 {\rm CC}_m^\bullet(G) = \coprod_{d|m,  d>1} {\rm CC}_{m,d}^\bullet(G).
\]
For each such $d$, we have an obvious well-defined surjective map ${\rm CC}_{m,d}^\bullet(G)  \to {\rm CC}_{d,d}^\bullet(G)$ which is injective if $m$ is a power of a prime. Thus identifying the two sets when $m=q^r$ with $q$ a prime number, we obtain a decomposition
 \[
 {\rm CC}_{q^r}^\bullet(G) = \coprod_{s=1}^r {\rm CC}_{q^s,q^s}^\bullet(G).
\]
The group we are interested in is ${\rm GL}_2(\mathbb Z/p\mathbb Z)$, that we identify with ${\rm Aut}(\mathbb Z_p^2)$, and for which we have the following result.

\begin{lemma}\label{lem:repautzp2}
Let $p,q$ be odd prime numbers and let $r\geq 1$ be such that $q^r|(p-1)$. Let $\xi$ be a  root of unity of order $q^r$ in $\mathbb Z/p\mathbb Z$. The set 
\[
\left\{
\begin{pmatrix}
          \xi & 0 \\
          0 & \xi^l  
         \end{pmatrix}, \ l \in \{1, 2, \ldots, \frac{q^r-1}{2}, q^r-1\}, \ {\rm GCD}(q,l)=1 \right\} \cup
\left\{
\begin{pmatrix}
          \xi & 0 \\
          0 & \xi^{qu}
         \end{pmatrix}, \ 0 \leq u < q^{r-1}
\right\}
\]
is a set of representatives for the elements of ${\rm CC}_{q^r,q^r}^\bullet({\rm GL}_2(\mathbb Z/p\mathbb Z)).$
\end{lemma}

The proof is a direct verification, using the fact that elements of order $q^r$ in  ${\rm GL}_2(\mathbb Z/p\mathbb Z)$ are diagonalizable. We now discuss when the above automorphisms are part of reduced $q^r$-data.

\begin{lemma}\label{lem:datazp2}
 Let $p,q$ be odd prime numbers and let $r\geq 1$ be such that $q^r|(p-1)$. Let $\theta$ be an automorphism of order $q^r$ of $\mathbb Z_p^2$, represented by one of the matrices of the previous lemma.
\begin{enumerate}
 \item If $\theta = \begin{pmatrix}
          \xi & 0 \\
          0 & \xi^{-1}  
         \end{pmatrix}$, there does not exist any reduced $q^r$-datum having $\theta$ as underlying automorphism.
\item Otherwise, there exists, up to equivalence, exactly one reduced $q^r$-datum having  $\theta$ as underlying automorphism.
\end{enumerate}
\end{lemma}

\begin{proof}
Fix generators $x_1$, $x_2$ of $\mathbb Z_p^2$, and  for $\omega\in \mu_p$, let $\tau_\omega : \mathbb Z_p^2\times \mathbb Z_p^2 \to k^\times$ be the unique bicharacter such that 
\[
 \tau_\omega(x_1,x_1)=1=\tau_\omega(x_2,x_2)=\tau_\omega(x_2,x_1),  \ \tau_\omega(x_1,x_2)=\omega. 
\]
It is well-known that any $2$-cocycle on $\mathbb Z_p^2$ is cohomologous to $\tau_\omega$ for some $\omega  \in \mu_p$, and that such $2$-cocycles $\tau_\omega$ and $\tau_{\omega'}$ are cohomologous if and only if $\omega=\omega'$. By Corollary \ref{cor:basiciso}, we can assume that any $q^r$-datum has $\tau_\omega$ as underlying cocycle, for some $\omega \in \mu_p$. Moreover a direct computation gives, for $\omega\not=1$ 
\[
 \prod_{k=0}^{q^r-1} \tau_\omega\circ \theta^k\times \theta^k =1 \iff  \prod_{k=0}^{q^r-1} [\tau_\omega\circ \theta^k\times \theta^k] =1 \iff \theta \not= \begin{pmatrix}
          \xi & 0 \\
          0 & \xi^{-1}  
         \end{pmatrix}
\]
and this shows the first assertion, since a datum is not reduced if the underlying cocycle is trivial. Moreover, for any $a\in \widehat{\mathbb Z_p^2}$ such that $a\circ \theta=a$, we obtain a reduced $q^r$-datum $(\mathbb Z_p^2,\theta,a, \tau_\omega)$, and any reduced $q^r$-datum arises in this way. 

We can now prove the second assertion via a case by case discussion. If $\theta = \begin{pmatrix}
          \xi & 0 \\
          0 & \xi^i
         \end{pmatrix}$ with $\xi^i\not=1$, the only compatible $a$ is $a=1$. We then see that, for $1\leq k\leq  p-1$, the $q^r$-data
\[
 (\mathbb Z_p^2, \theta, 1, \tau_{\omega}) \ {\rm and} \  (\mathbb Z_p^2, \theta, 1, \tau_{\omega^k})
\]
are equivalent, using, in the notation of Definition \ref{def:equivdata}, $f^{-1}=\begin{pmatrix} k& 0 \\ 0 & 1 \end{pmatrix}$ and $\varphi=1$.
 
 If $\theta = \begin{pmatrix}
          \xi & 0 \\
          0 & 1
         \end{pmatrix}$, the compatible $a$'s are given by $a(x_1)=1$ and $a(x_2)=\omega^k$, $0\leq k\leq p-1$. Denote by $a_{k}$ such an element of $\widehat{\mathbb Z_p^2}$.
We then see that, for $0\leq k_1\leq  p-1$ and $1\leq k_1\leq p-1$, the $q^r$-data
\[
 (\mathbb Z_p^2, \theta, 1, \tau_{\omega}) \ {\rm and} \  (\mathbb Z_p^2, \theta, a_{k_1}, \tau_{\omega^{k_2}})
\]
are equivalent, using, in the notation of Definition \ref{def:equivdata}, $f^{-1}=\begin{pmatrix} k_2& 0 \\ 0 & 1 \end{pmatrix}$ and $\varphi \in \widehat{\mathbb Z_p^2}$ such that $\varphi(x_1)=1$ and 
$\varphi(x_2)^{q^r}=\omega^{-k_1}$. This concludes the proof.
\end{proof}



\begin{proof}[Proof of Theorem \ref{thm:p2qr}]
 Let $A$ be a Hopf algebra as in the statement of Theorem \ref{thm:p2qr}: there exists a $q^r$-datum $(\mathbb Z_p^2,\theta, a, \tau)$ such that $A\simeq A_{q^r}(\mathbb Z_p^2,\theta, a , \tau)$, and with $\theta\not={\rm id}$ and $[\tau]\not=1$ (Proposition \ref{prop:a(h)co}), and the datum is reduced, as we have seen in the proof of the previous lemma.
Hence, by Proposition \ref{prop:isoa(h)} (and Theorem \ref{thm:isoext}) we have a bijection between isomorphism classes of Hopf algebras $A$ as above  and equivalence classes of $q^r$-data  $(\mathbb Z_p^2,\theta, a, \tau)$, $\theta \not={\rm id}$, $[\tau]\not=1$. For $1\leq s\leq r$, let $\mathcal E_s$ be the set  of equivalence classes of $q^r$-data as above and with $\theta$ of order $q^s$. Clearly $\mathcal E = \coprod_{s=1}^r\mathcal E_s$. Using Corollary \ref{cor:basiciso}, 
Lemma \ref{lem:repautzp2} and Lemma \ref{lem:datazp2}, we obtain  $|\mathcal E_s|=\frac{q^s+q^{s-1}}{2}$ for each $1\leq s\leq r$, and the announced result follows.
\end{proof}

The above reasoning works as well when $q=2$, the only small difference being in the counting process of Lemma \ref{lem:repautzp2}. The result is as follows (again, when $r=1$, this was proved in \cite{nat}). 

\begin{theorem}\label{thm:p22r}
Let $p$ be an odd prime, let $r\geq 1$ and assume that $2^{r}|p-1$. 
 The number of isomorphism classes of noncommutative and noncocommutative Hopf algebras fitting into a cocentral extension $k \to \Oo(\mathbb Z_p^2) \to A \to k\mathbb Z_{2^r} \to k$ is $1$ if $r=1$, and is $2(3.2^{r-2}-1)$ if $r\geq 2$.
\end{theorem}



\begin{thebibliography}{25}
 
 
 
\bibitem{ad}
N.~Andruskiewitsch, J.~Devoto, 
Extensions of Hopf algebras, 
 \emph{St.\ Petersburg Math.\ J.}  {\bf 7} (1996), no. 1, 17-52.

\bibitem{am}
N. Andruskiewitsch, M. M\"uller,  Examples of extensions of Hopf algebras, \emph{Rev. Colombiana Mat.} {\bf 49} (2015), no. 1, 193-211. 

\bibitem{bina}
J. Bichon, S. Natale, Hopf algebra deformations of binary polyhedral groups, \emph{Transform. Groups} {\bf 16} (2011), no. 2, 339-374.

\bibitem{bny16}  J. Bichon, S. Neshveyev, M. Yamashita, Graded twisting of categories and quantum groups by group actions, \emph{Ann. Inst. Fourier (Grenoble)} {\bf 66} (2016), no. 6, 2299-2338.

\bibitem{bny18}
J. Bichon, S. Neshveyev, M. Yamashita, Graded twisting of comodule algebras and module categories, \emph{J. Noncommut. Geom.} {\bf 12} (2018), no. 1, 331-368. 

\bibitem{chi}
A. Chirvasitu,  Centers, cocenters and simple quantum groups, \emph{J. Pure Appl. Algebra} {\bf 218} (2014), no. 8, 1418-1430. 

\bibitem{chka} A. Chirvasitu, P. Kasprzak, On the Hopf (co)center of a Hopf algebra, \emph{J. Algebra} {\bf 464} (2016), 141-174.

\bibitem{dieu}  J. A. Dieudonn\'e,  La g\'eom\'etrie des groupes classiques. Troisi\`eme \'edition. Ergebnisse der Mathematik und ihrer Grenzgebiete, Band 5. Springer-Verlag, Berlin-New York, 1971.

\bibitem{doi}
Y. Doi, Braided bialgebras and quadratic bialgebras, \emph{Comm. Algebra} {\bf 21} (1993), no. 5, 1731-1749.


\bibitem{fuk}
N. Fukuda,  Semisimple Hopf algebras of dimension 12, \emph{Tsukuba J. Math.} {\bf 21} (1997), no. 1, 43-54.

\bibitem{gm}
C. Galindo, Y. Morales,  A five-term exact sequence for Kac cohomology, \emph{Algebra Number Theory} {\bf 13} (2019), no. 5, 1121-1144.



\bibitem{hs}
P.J. Hilton, U. Stammbach, A course in homological algebra. Second edition. Graduate Texts in Mathematics, 4. Springer-Verlag, 1997.

\bibitem{kac}
G. I. Kac, Extensions of groups to ring groups, \emph{Math. USSR, Sb.} {\bf 5} (1969), 451-474.

\bibitem{kar}
G. Karpilovsky, The Schur multiplier. London Mathematical Society Monographs. New Series, 2. The Clarendon Press, Oxford University Press, New York, 1987.

\bibitem{kas00}
Y. Kashina, Classification of semisimple Hopf algebras of dimension 16, \emph{J. Algebra} {\bf 232} (2000), no. 2, 617-663.
 
\bibitem{kas03}
Y. Kashina,  On semisimple Hopf algebras of dimension $2^m$, \emph{Algebr. Represent. Theory} {\bf 6} (2003), no. 4, 393-425.

\bibitem{kas16}
Y. Kashina,  On semisimple Hopf algebras of dimension $2^m$, II, \emph{Algebr. Represent. Theory} {\bf 19} (2016), no. 6, 1387-1422.

\bibitem{kw}
D. Kazhdan, H. Wenzl, 
Reconstructing monoidal categories, I. M. Gelfand Seminar, 111-136,
\emph{Adv. Soviet Math.}, {\bf 16}, Part 2, Amer. Math. Soc., Providence, RI, 1993. 

\bibitem{kro}
 L. Krop, On the classification of finite-dimensional semisimple Hopf algebras,  \emph{Contemp. Math.} {\bf 688} (2017), 181-218.


\bibitem{mas95}
A. Masuoka, Self-dual Hopf algebras of dimension $p^3$ obtained by extension, \emph{J. Algebra} {\bf 178} (1995), no. 3, 791-806.

\bibitem{mas97}
A. Masuoka,  Calculations of some groups of Hopf algebra extensions, \emph{J. Algebra} {\bf 191} (1997), no. 2, 568-588.

\bibitem{mas}
A. Masuoka, Extensions of Hopf algebras, \emph{Trab. Math. (FAMAF)} {\bf 31}, 1999.

\bibitem{md}
A. Masuoka, Y. Doi,  Generalization of cleft comodule algebras, \emph{Comm. Algebra} {\bf 20} (1992), no. 12, 3703-3721.

\bibitem{mon} S. Montgomery, 
Hopf algebras and their actions on rings, Amer. Math. Soc. 1993.


\bibitem{nat}
S. Natale, On semisimple Hopf algebras of dimension $pq^2$, \emph{J. Algebra} {\bf 221} (1999), no. 1, 242-278.

\bibitem{ny}
S. Neshveyev, M.  Yamashita, Twisting the q-deformations of compact semisimple Lie groups, \emph{J. Math. Soc. Japan} {\bf 67} (2015), no. 2, 637-662. 

\bibitem{pod}
P. Podle\'s,  Symmetries of quantum spaces. Subgroups and quotient spaces of quantum SU(2) and SO(3) groups, \emph{Comm. Math. Phys.} {\bf 170} (1995), no. 1, 1-20.


\bibitem{sch}
H.J. Schneider, Normal basis and transitivity of crossed products for Hopf algebras. \emph{J. Algebra} {\bf 152} (1992), no. 2, 289-312.

\bibitem{zha}
J. J. Zhang, Twisted graded algebras and equivalences of graded categories, \emph{Proc. London Math. Soc.} {\bf 72} (1996), no. 2, 281-311. 

\end{thebibliography}
\end{document}